\documentclass{pseudolms}
\usepackage[english,british]{babel}
\usepackage{amsmath,amssymb,bbm,textcomp,graphicx,latexsym,xcolor,paralist}%,ntheorem,enumerate,theorem,amsbsy,
\usepackage[T1]{fontenc}%T3,T1 dla "tipa"
\usepackage{mathptmx}
\usepackage[activate={true,nocompatibility},final,tracking=true,kerning=true,spacing=true,factor=1100,stretch=10,shrink=10]{microtype}%³adniejszy sk³ad
\usepackage{fancyhdr}
\usepackage[pdftex,pdfborderstyle={/S/U/W 1},hyperfootnotes=false,hidelinks]{hyperref}
\usepackage{mathtools}
\mathtoolsset{centercolon}%:=
\newcommand{\assign}{:=}
\newcommand{\divides}{\mathrel{|}}
\newcommand{\nequiv}{\mathrel{\not\equiv}}
\newcommand{\nospace}{}

\newcommand{\tmem}[1]{{\em #1\/}}
\newcommand{\tmname}[1]{\textsc{#1}}
\newcommand{\tmop}[1]{\ensuremath{\operatorname{#1}}}
\newcommand{\tmstrong}[1]{\textbf{#1}}
\newcommand{\tmtextbf}[1]{{\bfseries{#1}}}
\newcommand{\tmtextit}[1]{{\itshape{#1}}}

\newenvironment{enumeratealpha}{\begin{enumerate}[a{\textup{)}}] }{\end{enumerate}}
\newenvironment{enumerateroman}{\begin{enumerate}[{\textup{(}}i{\textup{)}}] }{\end{enumerate}}

\newnumbered{remark}{Remark}
\newnumbered{comment}{Comment}
\newnumbered{comments}{Comments}
\newnumbered{definition}{Definition}
\newnumbered{notation}{Notation}
\newnumbered{example}{Example}

\newtheorem{theorem}{Theorem}
\newtheorem{lemma}{Lemma}
\newtheorem{proposition}{Proposition}
\newtheorem{corollary}{Corollary}
\simpleequations

\title[The {\L}oj.\ exponent of non-degenerate surface singularities]{The {\L}ojasiewicz exponent of non-degenerate surface singularities}

\author[S.\ Brzostowski, T.\ Krasi\'nski, and G.\ Oleksik]{
  Szymon Brzostowski,
  Tadeusz Krasi{\'n}ski,
  \and
  Grzegorz Oleksik
}

\classno{32S05, 32S25, 14B05, 14J17}
\date{April 8, 2020}

\begin{document}

\maketitle

\begin{abstract}
 {Let $f$ be an isolated singularity at the origin of
$\mathbb{C}^n$. One of many invariants that can be associated with $f$ is its
{\L}ojasiewicz exponent $\mathcal{L}_0 (f)$, which measures, to some extent,
the topology of $f$. We give, for generic surface singularities $f$, an
effective formula for $\mathcal{L}_0 (f)$ in terms of the Newton polyhedron of
$f$. This is a realization of one
of Arnold's postulates.}
\end{abstract}

%\tmmsc{}

%\tmkeywords{}

\section{Introduction}

Let $f (z) = f (z_1, \ldots, z_n) \in \mathbb{C} \{z_1, \ldots, z_n \} = :
\mathcal{O}_n$ be a convergent power series defining an {\tmem{isolated
singularity }}at the origin $0 \in \mathbb{C}^n$, i.e. $f (0) = 0$ and the
gradient of $f$,
\[ \nabla f \assign \left( \frac{\partial f}{\partial z_1}, \ldots, \frac{\partial f}{\partial
z_n} \right) : (\mathbb{C}^n, 0) \rightarrow (\mathbb{C}^n, 0), \]
has an isolated zero at $0 \in \mathbb{C}^n$. The {\tmem{{\L}ojasiewicz
exponent }}$\mathcal{L}_0 (f)$ of $f$ is the smallest $\theta > 0$ such that
there exists a neighbourhood $U$ of $0 \in \mathbb{C}^n$ and a constant $C >
0$ such that
\[ | \nabla f (z) | \geqslant C | z |^{\theta} \text{ for } z \in U. \]
\begin{remark*}
  One can similarly define the {\L}ojasiewicz exponent{\tmem{ }}$\mathcal{L}_0
  (F)$ of any holomorphic mapping $F : (\mathbb{C}^n, 0) \rightarrow
  (\mathbb{C}^p, 0)$ having an isolated zero at $0 \in \mathbb{C}^n$.
\end{remark*}
It is known $\mathcal{L}_0 (f)$ is a rational positive number, it is an
analytic invariant of $f$, it depends only on the ideal $(\frac{\partial
f}{\partial z_1}, \ldots, \frac{\partial f}{\partial z_n})$ in
$\mathcal{O}_n$, $[\mathcal{L}_0 (f)] + 1$ is the degree of $C^0$-sufficiency
of $f$ and it can be calculated by means of analytic paths, i.e.
\begin{equation}
  \mathcal{L}_0 (f) = \sup_{\Phi}  \frac{\tmop{ord} (\nabla f \circ
  \Phi)}{\tmop{ord} \Phi}, \label{wyk_par}
\end{equation}
where $0 \neq \Phi = (\varphi_1, \ldots, \varphi_n) \in \mathbb{C} \{t\}^n$,
$\Phi (0) = 0$, and $\tmop{ord} \Phi \assign \min_i \tmop{ord} \varphi_i$. It
is an open and difficult problem if the {\L}ojasiewicz exponent is a
topological invariant. There are many
explicit formulas and estimations for $\mathcal{L}_0 (f)$ in various terms and
in special classes of singularities (see {\cite{KL77}}, {\cite{LT08}},
{\cite{Tei77a}}, {\cite{Brz15}}, {\cite{KOP09}}).

In the paper, we investigate the problem of determining the {\L}ojasiewicz
exponent for non-degenerate (in the Kushnirenko sense) singularities in terms
of their Newton polyhedron. This is an answer to V. Arnold's problem who
postulated, ``every interesting discrete invariant of a generic singularity
with Newton polyhedron $\Gamma$ is an interesting function of the polyhedron''
(Problem 1975-1, see also 1968-2, 1975-21 in {\cite{Arn04}}). The most famous
example of such an invariant is the Milnor number, which can be calculated
using the Kushnirenko formula in the case of non-degenerate singularities (see
{\cite{Kou76}}). Here, we completely solve the above Arnold's problem for the
{\L}ojasiewicz exponent in the class of non-degenerate surface singularities,
i.e.\ for $n = 3$. The case $n = 2$ was solved by Lenarcik in {\cite{Len98}}. Other
attempts to read off the {\L}ojasiewicz exponent from the Newton polyhedron of
a singularity were made by many authors (e.g.\ Lichtin {\cite{Lic81}}, Fukui
{\cite{Fuk91}}, {Bivi{\`a}-Ausina {\cite{Biv03}}},
Abderrahmane {\cite{Abd05}}, P.\ Mondal (private communication)). In
particular, see the recent paper by Oka {\cite{Oka18}}, who, in the
$n$-dimensional case, obtained estimations from above under
additional assumptions.

Our result is a generalization of the Lenarcik's one. Among segments of the
Newton polygon of $f$, he distinguished {\tmem{exceptional }}ones. The
{\L}ojasiewicz exponent, then, turns out to be the greatest coordinate of
intersections of the prolongations of the non-exceptional segments with the
coordinate axes. Our approach is similar: based on the definition, given by
the third-named author in {\cite{Ole13}}, we also distinguish exceptional
faces in the Newton polyhedron of a surface singularity. The greatest
coordinate of intersections of the prolongations of the non-exceptional
$2$-dimensional faces with the coordinate axes is exactly the {\L}ojasiewicz
exponent of the singularity.
The proof is based on the existence of polar curves of $f$ associated with
some faces of the Newton polyhedron of $f$.

\section{The main theorem}

Let $0 \neq f : (\mathbb{C}^n, 0) \rightarrow (\mathbb{C}, 0)$ be a
holomorphic function defined by a convergent power series $\sum_{\nu \in
\mathbb{N}^n} a_{\nu} z^{\nu}$, $z = (z_1, \ldots, z_n)$. Let $\mathbb{R}_+^n
\assign \{(x_1, \ldots, x_n) \in \mathbb{R}^n : x_i \geqslant 0, i = 1, \ldots,
n\}$. We define $\tmop{supp} f \assign \{\nu \in \mathbb{N}^n : a_{\nu} \neq
0\} \subset \mathbb{R}_+^n$. In the sequel, we will identify $\nu = (\nu_1,
\ldots, \nu_n) \in \tmop{supp} f$ with their associated monomials $z^{\nu} =
z_1^{\nu_1} \cdots z_n^{\nu_n}$. In particular, we will apply the geometric
language of the space $\mathbb{R}_+^n$ to polynomials, e.g. we may say $z
{}^{\nu} \in \tmop{supp} f$ or that two binomials are parallel (the latter
means segments in $\mathbb{R}_+^n$ associated to these binomials are
parallel). We define the {\tmem{Newton polyhedron}} $\Gamma_+ (f) \subset
\mathbb{R}_+^n$ {\tmem{of}} $f$ as the convex hull of $\{\nu +\mathbb{R}_+^n :
\nu \in \tmop{supp} f\}$. We say $f$ is {\tmem{convenient}} if $\Gamma_+ (f)$
has a non-empty intersection with each coordinate axis $0 x_i$ $(i = 1,
\ldots, n)$, and {\tmem{nearly convenient}} if, for each $i = 1, \ldots, n$,
in $\tmop{supp} f$ there is a monomial of the form $z_i^m$ or $z_i^m z_j$ with
$i \neq j$.
% For a convenient $f$, we define the compact polyhedron{$\Gamma_- (f)
% \assign \overline{\mathbb{R}_+^n \setminus \Gamma_+ (f)}$}.
Let $\Gamma (f)$,
the {\tmem{Newton boundary of $f$}}, be the set of compact boundary faces of
$\Gamma_+ (f)$, of any dimension. Denote by $\Gamma^k (f)$ the set of all
$k$-dimensional faces of $\Gamma (f)$ $(k = 0, \ldots, n - 1)$. Then $\Gamma
(f) = \bigcup_k \hspace{0.17em} \Gamma^k (f)$. For each (compact) face $S \in
\Gamma (f)$ we define the quasihomogeneous polynomial $f_S \assign \sum_{\nu
\in S} a_{\nu} z^{\nu}$. We say $f$ is {\tmem{$\mathcal{K}\!\!$-non-degenerate on
$S$}} (i.e.\ {\tmem{non-degenerate in the Kushnirenko
sense on $S$}}) if the system of polynomial equations $\frac{\partial f_S}{\partial z_i}
= 0$ $(i = 1, \ldots, n)$ has no solutions in $(\mathbb{C}^{\ast})^n$; $f$ is
{\tmem{$\mathcal{K}\!\!$-non-degenerate}}
%(i.e.\ {\tmem{non-degenerate in the Kushnirenko sense) }} 
if $f$ is $\mathcal{K}\!\!$-non-degenerate on each face $S \in \Gamma
(f)$.

For each $(n - 1)$-dimensional face $S \in \Gamma^{n - 1} (f)$ there exists a
vector $\mathbf{v}_S = (v_1, \ldots, v_n) \in \mathbb{Q}_{> 0}^n$ with
positive rational coordinates, perpendicular to $S$. It is called a {\tmem{normal vector of}}
$S$ and it is unique up to positive scaling. Then $S = \{x \in \Gamma_+ (f)
: \langle \mathbf{v}_S, x \rangle = l_S \}$, where $l_S \assign \inf \{\langle
\mathbf{v}_S, x \rangle : x \in \Gamma_+ (f)\}$. The unique hyperplane $\Pi_S$
containing $S$ has the equation $\langle \mathbf{v}_S, x \rangle = l_S$ and it
is a supporting hyperplane of $\Gamma_+ (f)$. Since $\mathbf{v}_S$ has
positive coordinates, $\Pi_S$ intersects each coordinate axis $0 x_i$ $(i = 1,
\ldots, n)$ in a point whose distance from $0$ is equal to $m (S)_{x_i} \assign
\frac{l_S}{v_i} > 0$. We define
\[ m (S) \assign \max_i  \{ m (S)_{x_i} \} = \max_i  \left\{ \frac{l_S}{v_i}
   \right\} . \]
   On the other hand, if $\mathbf{v} \in \mathbb{Q}_{> 0}^n$, we define 
   $l_\mathbf{v} \assign \inf \{\langle
\mathbf{v}, x \rangle : x \in \Gamma_+ (f)\}$
and
$S_\mathbf{v} \assign \{x \in \Gamma_+ (f)
: \langle \mathbf{v}, x \rangle = l_\mathbf{v} \}$.
$S_\mathbf{v}$ is a face of $\Gamma(f)$
and $\Pi_\mathbf{v}\assign\{x \in \mathbb{R}^n
: \langle \mathbf{v}, x \rangle = l_\mathbf{v} \}$ is a \tmem{supporting hyperplane of\/ $\Gamma_+ (f)$ along $S_\mathbf{v}$}
(for short, $\mathbf{v}$ is called a \tmem{supporting vector along $S_\mathbf{v}$}).

We say $S \in \Gamma^{n - 1} (f)$ is {\tmem{exceptional with respect to the
axis}} $0 x_i$ if one of the partial derivatives $\frac{\partial f_S}{\partial
z_j}$, $j \neq i$, is a pure power of $z_i$. {Geometrically, this means $S$ is an
$(n - 1)$-dimensional pyramid with the
base lying in the $(n - 1)$-dimensional coordinate hyperplane $\{ x_j = 0 \}$, where $j \neq i$,
and with its apex lying at distance 1 from the axis $0 \nospace x_i$ in the
direction of $0x_j$;} see Figure \ref{fig_excep}. A face {\noindent}$S \in \Gamma^{n - 1}
(f)$ is {\tmem{exceptional }}if $S$ is exceptional with respect to some axis.
Let $\omega\in\{x,y,z\}$.
We denote the set of exceptional faces of $f$ with respect to $0\omega$ by $E_{\omega}(f)$. We denote the set of all exceptional faces of $f$ by $E (f)$.
 \filbreak%do zlamania strony

\begin{figure}[h!]
  \resizebox{!}{140pt}{\includegraphics{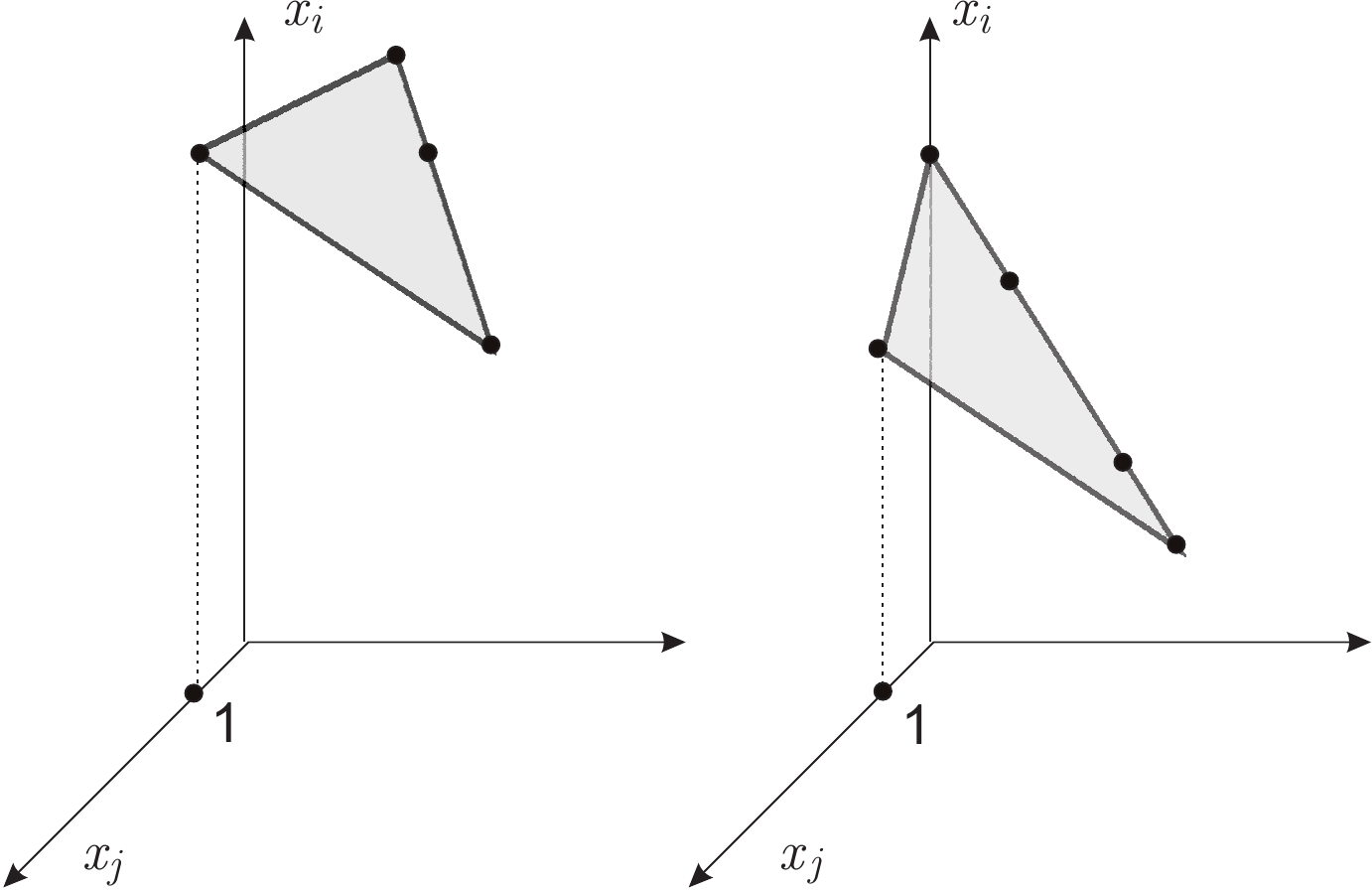}}
  \caption{\label{fig_excep}Exceptional faces with respect to the axis $0\nospace x_i$.}
	 
\end{figure}
\nopagebreak The main theorem is \nopagebreak[4]

 \nopagebreak\begin{theorem}
  \label{tw_main}If $f : (\mathbb{C}^3, 0) \rightarrow (\mathbb{C}, 0)$ is a
  $\mathcal{K}\!\!$-non-degenerate isolated surface singularity possessing non-exceptional
  faces, i.e.\/ $\Gamma^2 (f) \setminus E (f) \neq \varnothing$, then
  \begin{equation}
    \mathcal{L}_0 (f) = \max_{S \in \Gamma^2 (f) \setminus E (f)} m (S) - 1. \label{formula2}
  \end{equation}
\end{theorem}

\begin{comments*}
\begin{enumerate}
  \item \label{com_2}For $n = 3$, the case $\Gamma^2 (f) \setminus E (f) =
  \varnothing$ was established by the third-named author in {\cite[Theorem
  1.8]{Ole13}}. Namely, in this case, if we denote the variables in
  $\mathbb{C}^3$ by $x, y, z$ (and permute them if necessary), there is
  exactly one segment $S \in \Gamma^1 (f)$ joining the monomial $xy$ with some
  monomial $z^k$, where $k \geqslant 2$, and then $\mathcal{L}_0 (f) = k - 1$.
  \item The maximum in formula \eqref{formula2} may be calculated over the set of certain distinguished
non-exceptional faces (see Concluding remarks).
  \item Lenarcik in {\cite{Len98}} proved an alike formula for $n = 2$.
  Precisely,
  \[ \mathcal{L}_0 (f) = \left\{ \begin{array}{ll}
       \max_{S \in \Gamma^1 (f) \setminus E (f)} m (S) - 1, & \text{if }
       \Gamma^1 (f) \setminus E (f) \neq \varnothing\\
       1, & \text{if } \Gamma^1 (f) \setminus E (f) = \varnothing
     \end{array} \right. \!\!\!. \]
  \item Directly from the definitions, exceptional faces are always $\mathcal{K}\!\!$-non-degenerate.
  %This suggests that they should be irrelevant for the topology of $f$.
\end{enumerate}
\end{comments*}

As an immediate consequence of Theorem \ref{tw_main}, we get the following strengthening of
the main result of \cite{Brz19} (see Theorem \ref{Brzost} in the next section).
\begin{corollary}
  \label{le_remove} Let $f, g : (\mathbb{C}^3, 0) \rightarrow (\mathbb{C}, 0)$ be two
  $\mathcal{K}\!\!$-non-degenerate isolated surface singularities at $0\in\mathbb{C}^3$. Assume
  that $\varnothing\neq\Gamma^2 (f)\setminus E(f)=\Gamma^2 (g) \setminus E(g)$.
  Then
  \[ \mathcal{L}_0 (f) =\mathcal{L}_0 (g) . \]
\end{corollary}

\section{Preliminaries}

A crucial r{\^o}le in the proof of the main theorem is played by the following
result proved by the first-named author, which can be intuitively rephrased by
saying that, in the non-degenerate case, $\mathcal{L}_0 (f)$ depends only on
the Newton polyhedron of $f$. This observation considerably simplifies the
family of singularities we need to consider.

\begin{theorem}[({\cite{Brz19}})]
  \label{Brzost}If $f, g : (\mathbb{C}^n, 0) \rightarrow (\mathbb{C}, 0)$ are
  $\mathcal{K}\!\!$-non-degenerate isolated singularities and $\Gamma (f) = \Gamma
  (g)$, then $\mathcal{L}_0 (f) =\mathcal{L}_0 (g)$.
\end{theorem}

This theorem, probably known to experts, follows from general facts proved
independently by various authors, to mention Damon and Gaffney {\cite{DG83}},
Yoshinaga {\cite{Yos89}}, or Oka {\cite[{Theorem 5.3}]{Oka97}}.

In the proof of the main theorem, we will apply the celebrated Bernstein
theorem on the number of solutions to systems of polynomial equations under
some non-degeneracy conditions (so-called Bernstein non-degeneracy). To
formulate this theorem, we need to introduce some notions. Let $f \neq 0$ be a
polynomial in $n$ variables $z_1, \ldots, z_n$. For each vector $0 \neq
\mathbf{v} \in \mathbb{Z}^n$ we denote by $f_{\mathbf{v}}$ the {\tmem{initial part of
$f$}} (the part of the least degree) with respect to the weights of variables
given by the coordinates of $\mathbf{v}$. The {\tmem{Newton polytope}}
$\mathcal{N} (f)$ {\tmem{of}} $f$ is defined to be the convex hull of
$\tmop{supp} f$. Then, the well-known formula holds: $\mathcal{N} (fg)
=\mathcal{N} (f) +\mathcal{N} (g)$, where the $+$ denotes the Minkowski sum of
subsets of $\mathbb{R}^n$. The {\tmem{mixed volume of a system $(f_1, \ldots,
f_n)$ of non-zero polynomials}} (or {\tmem{of the Newton polytopes
$\mathcal{N} (f_1), \ldots, \mathcal{N} (f_n)$}}) is defined by the formula
\[ \tmop{MV} (f_1, \ldots, f_n) = \sum_{k = 1}^n (- 1)^{n - k} \sum_{I \subset
   \{1, \ldots, n\}, |I| = k} \tmop{vol}_n (\sum_{i \in I} \mathcal{N}(f_i)) .
\]
It is always a non-negative number. In the case $n = 2$, considered in this
paper, the above formula takes the simple form
\[ \tmop{MV} (f_1, f_2) = \tmop{vol}_2 (\mathcal{N}(f_1) +\mathcal{N}(f_2)) -
   \tmop{vol}_2 (\mathcal{N}(f_1)) - \tmop{vol}_2 (\mathcal{N}(f_2)) . \]
Using this formula, it is easy to characterize the pairs $(f_1, f_2)$ for
which $\tmop{MV} (f_1, f_2) = 0$ (a general such result one can find, e.g., in
{\cite[Theorem 5.1.7]{Sch93}}).

\begin{lemma}
  \label{lem}$\tmop{MV} (f_1, f_2) = 0$ if and only if either $\mathcal{N}
  (f_1)$ and $\mathcal{N} (f_2)$ are parallel segments or at least one of
  $\mathcal{N} (f_1)$, $\mathcal{N} (f_2)$ reduces to a point.
\end{lemma}

The system $(f_1, \ldots, f_n)$ of polynomials is
{\tmem{$\mathcal{B}$-non-degenerate}} (i.e.\ is {\tmem{non-degenerate in the
Bernstein sense}}) if for each vector $0 \neq \mathbf{v} \in \mathbb{Z}^n$ the
system of equations $\{(f_1)_{\mathbf{v}} = 0, \ldots, (f_n)_{\mathbf{v}} =
0\}$ has no solution in $(\mathbb{C}^{\ast})^n$.

\begin{theorem}[({\cite{Ber75}})]
  \label{ber}Let $f_1, \ldots, f_n \in \mathbb{C} [z_1, \ldots, z_n]$ be
  polynomials. If the mixed volume $\tmop{MV} (f_1, \ldots, f_n)$ is positive
  and the system $(f_1, \ldots, f_n)$ is $\mathcal{B}$-non-degenerate, then
  the system of equations $\{f_1 = 0, \ldots, f_n = 0\}$ has
  only isolated solutions in $(\mathbb{C}^{\ast})^n$ and their number, counted
  with multiplicities, is equal to $\tmop{MV} (f_1, \ldots, f_n)$.

\end{theorem}

\begin{remark*}
  In the original statement of the Bernstein theorem, $f_i$ may even be
  Laurent polynomials, i.e.\ elements of the ring $\mathbb{C} [z_1, z_1^{- 1},
  \ldots, z_n, z_n^{- 1}]$.
\end{remark*}

To prove the existence of appropriate paths in $(\mathbb{C}^3, 0)$ realizing
the {\L}ojasiewicz exponent, we will apply the following Maurer theorem, which
is a generalization of the classical Puiseux theorem from the $2$-dimensional
case to the $n$-dimensional one.

\begin{theorem}[({\cite{Mau80}})]
  \label{mau}Let $f_1, \ldots, f_{n - 1} \in \mathfrak{m}\mathcal{O}_n$ be
  holomorphic function germs defining an analytic curve in $(\mathbb{C}^n, 0)$,
  i.e.\ $\dim \{f_1 = \cdots = f_{n - 1} = 0\} = 1$. If\/ $\mathbf{v} \in
  \mathbb{N}_+^n$ is a vector with positive integer coordinates for which the
  initial part $g_{\mathbf{v}}$ of every $g \in (f_1, \ldots, f_{n - 1})
  \mathcal{O}_n$ is not a monomial,
  %(such vector\/ $\mathbf{v}$ is called
  %a {\tmem{tropism of the ideal $(f_1, \ldots, f_{n - 1}) \mathcal{O}_n$}}),
  then there exists a parametrization $0 \neq \Phi = (\varphi_1, \ldots,
  \varphi_n) \in \mathbb{C} \{t\}^n$, $\Phi (0) = 0$, of a branch of\/ $\{f_1 =
  \cdots = f_{n - 1} = 0\}$ such that\/ $(\tmop{ord} \varphi_1, \ldots,
  \tmop{ord} \varphi_n)$ is an integer multiple of $\mathbf{v}$.
\end{theorem}

Finally, the proof of the main theorem makes use of the following criterion
for a singularity to be isolated, proved by the first and third-named authors.

\begin{proposition}[({\cite[Thm.\ 3.1 and Prop.\ 2.6]{BO16}})]
  {}\label{BrOl} Assume that $f : (\mathbb{C}^3, 0) \rightarrow
  (\mathbb{C}, 0)$ is holomorphic, $\mathcal{K}\!\!$-non-degenerate, and\/ $\nabla f (0) = 0$.
  Then $f$ has an isolated singularity at\/ $0$ if, and only if, the following
  conditions hold:
  \begin{enumerate}
    \item \label{isol_3}$f$ is nearly convenient,
    \item \label{isol_4}$\tmop{supp} f$ has points in each of the coordinate
    planes $0 xy$, $0 xz$, and $0 yz$.
  \end{enumerate}

\end{proposition}

\section{Existence of parametrizations for a face}\label{Wide-kreska}
Here and below, in the context of three variables, we shall denote by $(x, y,
z)$ both the coordinates in $\mathbb{C}^3$ (the domain of singularities) and
coordinates in $\mathbb{R}^3$ (the space for the Newton polyhedrons). It will always be clear from the context to which space they are referring to.
Moreover, we let
\begin{equation}\label{notacja}
\overline{g} (y, z) \assign g (1, y, z) \in
\mathbb{C} [y, z],
\end{equation}
for any polynomial $g \in \mathbb{C} [x, y, z]$. Then,
obviously, $\overline{\frac{\partial g}{\partial y}} = \frac{\partial
\overline{g}}{\partial y}$ and $\overline{\frac{\partial g}{\partial z}} =
\frac{\partial \overline{g}}{\partial z}$.\label{kreska}

The main result of this section is the following theorem, which assures, under
certain non-degeneracy conditions on 
a two-dimensional face of an isolated surface singularity,
%both $S$ and $f_S$,
that one can find a
parametrization of a branch of a polar curve 
%$\{ \frac{\partial f}{\partial
%y} = 0, \frac{\partial f}{\partial z} = 0\}$ 
whose \tmem{initial exponent vector} (i.e.\ its vector of orders) is
perpendicular to the 
aforementioned face.
%face $S$.
\label{Dopisac ogólniejsze twierdzenie - dla podpierania}
This theorem may also be viewed as a generalization of the Newton-Puiseux theorem from $2$
to $3$ indeterminates in the case of polar curves.
% $\{ \frac{\partial f}{\partial
%y} = 0, \frac{\partial f}{\partial z} = 0\}$.

\begin{theorem}
  \label{tw_par}Let $f : (\mathbb{C}^3, 0) \rightarrow (\mathbb{C}, 0)$ be an
  isolated singularity, and let $S$ be a $2$-dimensional face of\/ $\Gamma (f)$
  with a normal vector\/ $\mathbf{v} \in \mathbb{Q}^3_{> 0}$. Assume that
  $\tmop{MV} (\frac{\partial \overline{f_S}}{\partial y}, \frac{\partial
  \overline{f_S}}{\partial z}) > 0$ and the pair $(\frac{\partial
  \overline{f_S}}{\partial y}, \frac{\partial \overline{f_S}}{\partial z})$ is
  $\mathcal{B}$-non-degenerate. Assume, moreover, that $f$ is
  $\mathcal{K}\!\!$-non-degenerate on all the edges of $S$. Then there exists a
  parametrization\/ $0 \neq \Phi = (\varphi_1, \varphi_2, \varphi_3) \in \mathbb{C}
  \{t\}^3$, $\Phi (0) = 0$, of a branch of the curve $\{ \frac{\partial
  f}{\partial y} = 0, \frac{\partial f}{\partial z} = 0\}$ such that
  $(\tmop{ord} \varphi_1, \tmop{ord} \varphi_2, \tmop{ord} \varphi_3)$ is proportional to\/
  $\mathbf{v}$.
\end{theorem}

To prove this theorem, we need two lemmas. The first one is a generalization
of a well-known, two-dimensional criterion of the Kushnirenko non-degeneracy which
reads: {\tmem{if $F \in \mathbb{C} [x, y] \setminus \mathbb{C}$ is a
quasihomogeneous polynomial (with positive rational weights), then $F$ is
$\mathcal{K}\!\!$-degenerate if, and only if, its partial derivatives have a
common non-monomial factor (equivalently: $F$ has a multiple non-monomial
factor)}}.

\begin{lemma}
  \label{le_degeneracy}Let $F \in \mathbb{C} [x, y, z] \setminus \mathbb{C}$
  be a quasihomogeneous polynomial (with positive rational weights). If at
  least two its partial derivatives have a common non-monomial factor and\/
  $\dim \mathcal{N} (F) = 2$, then $F$ is $\mathcal{K}\!\!$-degenerate on an edge
  of $\mathcal{N} (F)$.\label{W n-wymiarze raczej sprzeczno�� z
  generyczno�ci�}
\end{lemma}

In this lemma, as $F$ is quasihomogeneous, $\Gamma^2 (F)$ consists of only one $2$-dimensional face, say $S$.
Thus, $\Gamma^2 (F) = \{ S \}$, $\mathcal{N} (F) =
S$ and $F = F_S$. In the course of its proof, we find it convenient to project
the face $S$ to the plane $0 \nospace y \nospace z$. It is important to note
that, as $S$ has a normal vector with positive coordinates, this projection
$\tmop{pr}_{y, z}$, when restricted to $S$, is a bijection and respects the
edges of $S$: $T$ is an edge of S if, and only if, $\tmop{pr}_{y, z} (T)$ is
an edge of $\tmop{pr}_{y, z} (S)$.
% Moreover, for any $G \in \mathbb{C} [x, y,
% z]$ having its support contained in a face of a Newton polyhedron, it holds
% $\tmop{pr}_{y, z} (\mathcal{N} (G)) =\mathcal{N} (\overline{G})$.

\begin{proof*}[of Lemma~{\rm\ref{le_degeneracy}}]
  As above, we have $F = F_S$, where $S =\mathcal{N} (F)$. W may choose
  $\mathbf{u} = (u_x, u_y, u_z) \in \mathbb{N}_+^3$ perpendicular to $S$; then
  $F$ is quasihomogeneous with respect to the weights defined by $\mathbf{u}$.
  Up to renaming of the variables, we can write
  \[ \frac{\partial F}{\partial y} = PQ_1,\ \ \frac{\partial F}{\partial z} =
     PQ_2 \label{14} \]
  \text{in } $\mathbb{C} [x, y, z]$, where $P$ is not a monomial and $P, Q_1, Q_2 \in \mathbb{C} [x, y, z]$ are
  quasihomogeneous with respect to the same vector of weights $\mathbf{u}$.
  
 % We will show that $F$ can be brought to a special form. First, we %may assume
 % 
We will show that $F$ can be brought to a special form. First, we may assume
$F$ is homogeneous. In fact, if we define
 $\tilde{F} (x, y, z) \assign F
  (x^{u_x}, y^{u_y}, z^{u_z})$, then $\tilde{F}$ is a homogeneous polynomial,
  $\dim \mathcal{N} (\tilde{F}) = 2$, and the $\mathcal{K}\!\!$-non-degeneracies
  of $\tilde{F}$ on $\mathcal{N} (\tilde{F})$ and $F$ on $\mathcal{N}(F) = S$
  (and on their corresponding boundary edges) are equivalent. Moreover,
  $\frac{\partial \tilde{F}}{\partial y}$ and $\frac{\partial
  \tilde{F}}{\partial z}$ still have a non-monomial factor in common, viz., $P
  (x^{u_x}, y^{u_y}, z^{u_z})$.
  
  After this simplification, we return to the original notations. Thus, $F$ is
  homogeneous of a degree $\alpha > 0$, $\frac{\partial F}{\partial y} = P Q_1$,
  $\frac{\partial F}{\partial z} = PQ_2$ in $\mathbb{C} [x, y, z]$, $P, Q_1,
  Q_2$ are homogeneous, and $P$ is not a monomial. Putting $x = 1$, we get
  $\frac{\partial \overline{F}}{\partial y} = \overline{P} \, \overline{Q}_1$
  and $\frac{\partial \overline{F}}{\partial z} = \overline{P} \, 
  \overline{Q}_2$ in $\mathbb{C} [y, z]$. As $P$ is homogeneous and not a
  monomial, $\overline{P}$ is not a monomial, either. Exchanging
  $\overline{P}$ for its irreducible, non-monomial factor, we may additionally
  assume that $\overline{P}$ itself is irreducible in $\mathbb{C} [y, z]$.
  From the Freudenburg lemma ({\cite{Fre96}} or {\cite[Corollary
  3.2]{vdENT03}}), then, we infer that $\overline{F} = \overline{P}^2 R_0 +
  c$, where $R_0 \in \mathbb{C} [y, z]$ and $c \in \mathbb{C}$. Hence,
  \[ F (x, y, z) = x^{\alpha}  \overline{F} (\frac{y}{x}, \frac{z}{x}) = P^2
     (x, y, z) R (x, y, z) + cx^{\alpha}, \label{15} \]
  where the homogeneous $R \in \mathbb{C} [x, y, z]$ satisfies $\overline{R} =
  R_0$.
  
  Passing to the proof of the degeneracy of $F$, we may assume $c \neq 0$; for
  otherwise, the assertion is clear. Utilizing the form of $F$, we shall
  indicate an edge $T$ of $S$ for which $F_T$ has a double, non-monomial
  factor;
  %this will be the source of the required degeneracy of $F$.
  then $F$ will be $\mathcal{K}\!\!$-degenerate on $T$. More
  precisely, we shall find an edge $T_0$ of $\mathcal{N} (P)$ whose supporting
  vector $\mathbf{v}$ defines an edge $T=T_\mathbf{v}$ of $S$ not passing through the point
  $x^{\alpha}$. Then $F_T = F_{\mathbf{v}} = (P^2 \cdot R)_{\mathbf{v}} =
  (P_{T_0})^2 \cdot R_{\mathbf{v}}$ and $F$ is $\mathcal{K}\!\!$-degenerate on $T
  \in \Gamma^1 (F)$. This goal can be achieved by projecting $S$ to the plane
  $0 \nospace y \nospace z$ and finding there an edge $\tilde{T}_0$ of
  $\mathcal{N} (\overline{P})$ whose supporting vector
  $\tilde{\mathbf{v}}$ defines an edge $\tilde{T}$ of $\overline{S}
  \assign \tmop{pr}_{y, z} (S)$ not passing through the point $(0, 0)$.
  Indeed, having such a $\tilde{T}$ and $\tilde{\mathbf{v}}$, it is enough
  to set $T \assign \tmop{pr}_{y, z}^{- 1} (\tilde{T}) \cap S$ and $\mathbf{v}
  \assign (0, \tilde{v}_1, \tilde{v}_2)$, and then the above calculation for
  $F_T$ applies.
  
  We have $\overline{F} \assign \overline{P}^2  \overline{R} + c$. As $S$ is a
  2-dimensional face, $\overline{S}$ is also $2$-dimensional. Similarly, as
  $P$ is not a monomial, $\mathcal{N} (\overline{P})$ is at least a segment. It is
  convenient to distinguish two possibilities on the shape of the Newton
  polytope $\mathcal{N} (\overline{P})$ of $\overline{P}$:
  
  {\underline{1st case}}: {\tmem{$\mathcal{N} (\overline{P})$ is a segment}}.
  Let $\mathbf{{\tmstrong{w}}} \in \mathbb{Z}^2 \backslash \{ 0 \}$ be a
  vector perpendicular to $\overline{P}$. If $\tmop{ord}_{\mathbf{w}} 
  (\overline{P}^2  \overline{R}) < 0$, then $\overline{F}_{\mathbf{w}} =
  (\overline{P}^2  \overline{R})_{\mathbf{w}} = \overline{P}^2 
  \overline{R}_w$. Hence, we may set $\tilde{T}_0 \assign \mathcal{N}
  (\overline{P})$ and $\tilde{\mathbf{v}} \assign \mathbf{w}$ in this
  case.  If $\tmop{ord}_{\mathbf{w}}  (\overline{P}^2  \overline{R}) \geqslant 0$,
  then we have $\tmop{ord}_{- \mathbf{w}}  (\overline{P}^2  \overline{R} + c)
  = - \deg_{\mathbf{w}} (\overline{P}^2  \overline{R} + c) \leqslant -
  \tmop{ord}_{\mathbf{w}} (\overline{P}^2  \overline{R} + c) \leqslant 0$. But then
  we must have $\tmop{ord}_{- \mathbf{w}}  (\overline{P}^2  \overline{R}) < 0$
  for, otherwise, $\overline{S}$ would be a segment. Now, as above, we may take
  $\tilde{T}_0 \assign \mathcal{N} (\overline{P})$ and $\tilde{\mathbf{v}}
  \assign - \mathbf{w}$.
  
  {\underline{2nd case}}: {\tmem{$\mathcal{N} (\overline{P})$ is a
  2-dimensional polygon}}. Let $T_1, \ldots, T_k$ be boundary edges of
  $\mathcal{N} (\overline{P})$ and $\mathbf{v}_1, \ldots, \mathbf{v}_k \in
  \mathbb{Z}^2 \backslash \{ 0 \}$, respectively, their (inward-pointing)
  normal vectors (so that $\overline{P}_{\mathbf{v}_i} = \overline{P}_{T_i}$); see Figure \ref{fig_minkowski} (a). These vectors also define the
  corresponding edges of $\mathcal{N} (\overline{P}^2  \overline{R})$, namely
  $\tilde{T}_i \assign \mathcal{N} \left( (\overline{P}^2 
  \overline{R})_{\mathbf{v}_i} \right)$ $(i = 1, \ldots, k)$. Clearly,
  $\tilde{T}_1, \ldots, \tilde{T}_k$ are parallel to $T_1, \ldots, T_k$,
  respectively. Since $\mathbf{v}_1, \ldots, \mathbf{v}_k$ is a collection of
  normal vectors of all the edges of a 2-dimensional convex polygon
  $\mathcal{N} (\overline{P})$ and the polygon $\mathcal{N} (\overline{P}^2 
  \overline{R})$ is also convex, the edges $\tilde{T}_1, \ldots, \tilde{T}_k$
  of $\mathcal{N} (\overline{P}^2  \overline{R})$ can be prolonged to form a
  convex polygon, too -- see Figure \ref{fig_minkowski} (b).
  
  \begin{figure}[h]
    \resizebox{314pt}{!}{\includegraphics{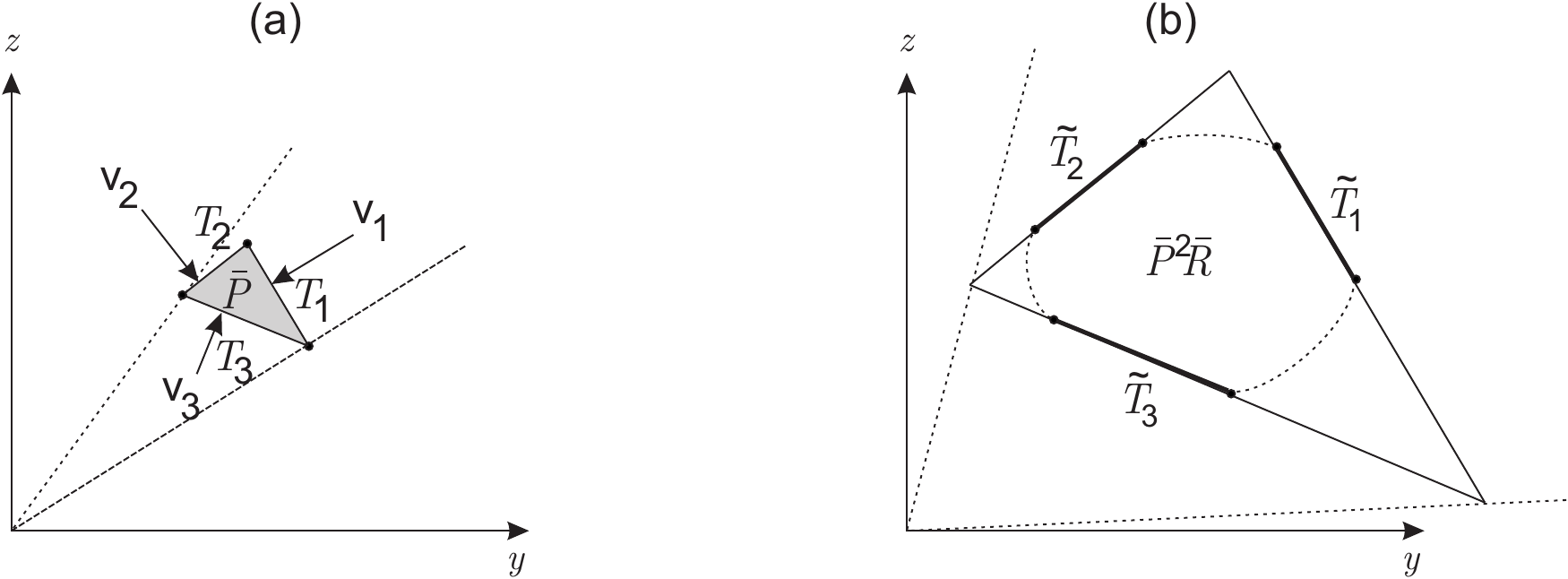}}
    \caption{\label{fig_minkowski}(a) The polygon $\mathcal{N}
    (\overline{P})$.\quad (b) The polygon $\mathcal{N} (\overline{P}^2 
    \overline{R})$ and prolongations of $\tilde{T}_i$.}
  \end{figure}
  
  But $\overline{S} =\mathcal{N} (\overline{P}^2  \overline{R} + c) =
  \tmop{conv} (\mathcal{N}(\overline{P}^2  \overline{R}) \cup \{(0, 0)\})$, so
  at least one of the edges $\tilde{T}_1, \ldots, \tilde{T}_k$ of $\mathcal{N}
  (\overline{P}^2  \overline{R})$ is also an edge of $\overline{S}$, one not
  passing through the point $(0, 0)$. Say it is $\tilde{T}_1$. Then we may set
  $\tilde{T} \assign \tilde{T}_1$ and $\tilde{\mathbf{v}} \assign
  \mathbf{v}_1$ in this case.\hspace*{\fill}\proofbox
\end{proof*}

\begin{remark*}
  It would be interesting to know if the above lemma holds, \tmem{mutatis mutandis}, in the
  $n$-dimensional case.
\end{remark*}

The second lemma is of algebraic nature.

\begin{lemma}
  \label{ret}Let $R$ be a unique factorization domain, $F, G \in R [[X_1,
  \ldots, X_n]] \setminus \{ 0 \}$ and\/ $\mathbf{v} \in \mathbb{N}_+^n$, where
  $n \geqslant 3$. Assume that $F_{\mathbf{v}}, G_{\mathbf{v}}$ have no common
  divisor in $R [X_1, \ldots, X_n]$ except monomials, and the ideal $\left(
  F_{\mathbf{v}}, G_{\mathbf{v}} \right)$ in $R [X_1, \ldots, X_n]$ does not
  contain any monomial. Then the ideal $(F, G) \subset R [[X_1, \ldots, X_n]]$
  does not contain any element $H$ such that $H_{\mathbf{v}}$ is a monomial.
\end{lemma}

\begin{remark*}
  Here, just as everywhere else in the paper, we {\tmem{do not}} demand a
  monomial to be monic. However, the lemma is still valid, with the same
  proof, if we do impose this restriction in both the assumptions and in the
  assertion of the lemma (simultaneously).
\end{remark*}

\begin{proof*}[of Lemma~{\rm\ref{ret}}]
  First, let us introduce a piece of notation: let $\mathfrak{S} \assign R
  [[X_1, \ldots, X_n]]$ and, for an ideal $\mathcal{I} \subset \mathfrak{S}$,
  let $\tmop{in}_{\mathbf{v}} \mathcal{I} \assign \left\{ g_{\mathbf{v}} : g
  \in \mathcal{I} \right\} \mathfrak{S}$ be the $\mathbf{v}$-initial ideal of
  $\mathcal{I}$.
  
  To the contrary, say there exists a monomial $\omega \in
  \tmop{in}_{\mathbf{v}}  (F, G)$. Then
  \begin{equation}
    AF + BG = \omega + \text{h.o.t.}, \label{wz3}
  \end{equation}
  for some $A$, $B \in \mathfrak{S}$. Here, $A \neq 0 \neq B$ by the second
  assumption.
  
  As a first reduction, we may think $\tmop{ord}_{\mathbf{v}} F =
  \tmop{ord}_{\mathbf{v}} G$. For this, it is enough to consider $X^{\alpha}
  F$ and $X^{\beta} G$ instead of $F$ and $G$ (respectively), where
  $X^{\alpha}$ is a monomial appearing in $G_{\mathbf{v}}$ and $X^{\beta}$ is a
  monomial appearing in $F_{\mathbf{v}}$. Then, a relation of type {\eqref{wz3}} still holds
  with $\omega$ modified to $X^{\alpha + \beta} \omega$.
  
  Secondly, we may assume $F_{\mathbf{v}}$ and $G_{\mathbf{v}}$ to be
  co-prime. Indeed, write $F = F_{\mathbf{v}} + F'$, $G = G_{\mathbf{v}} + G'$
  and let $\varrho \assign \gcd \left( F_{\mathbf{v}}, G_{\mathbf{v}} \right)$
  be the common monomial divisor. Substituting  $\varrho^{v_i} X_i$ for $X_i$ ($i =
  1, \ldots, n$), where $\mathbf{v} = (v_1, \ldots, v_n)$, from {\eqref{wz3}}
  we get $\tilde{A}  \left( F_{\mathbf{v}} + \rho \tilde{F}' \right) +
  \tilde{B}  \left( G_{\mathbf{v}} + \rho \tilde{G}' \right) = \varrho^t 
  \left( \omega + \text{h.o.t.} \right)$, where $t = \deg_{\mathbf{v}} \omega$
  and $\tilde{A}, \tilde{B}, \tilde{F}', \tilde{G}' \in \mathfrak{S}$. By
  assumption, there are no monomials in the ideal $\left( F_{\mathbf{v}},
  G_{\mathbf{v}} \right)$; thus, $t > 0$. Hence, $\tilde{A}  \left(
  \frac{F_{\mathbf{v}}}{\varrho} + \tilde{F}' \right) + \tilde{B}  \left(
  \frac{G_{\mathbf{v}}}{\varrho} + \tilde{G}' \right) = \varrho^{t - 1} \omega
  + \text{h.o.t.}$. Replacing $F$ by $\frac{F_{\mathbf{v}}}{\varrho} +
  \tilde{F}'$, $G$ by $\frac{G_{\mathbf{v}}}{\varrho} + \tilde{G}'$, and
  $\omega$ by $\varrho^{t - 1} \omega$ we easily see that, still, all the
  assumptions of the lemma hold for the changed $F$, $G$ but now $\gcd \left(
  F_{\mathbf{v}}, G_{\mathbf{v}} \right) = 1$; moreover, we still have
  $\tmop{ord}_{\mathbf{v}} F = \tmop{ord}_{\mathbf{v}} G$ and relation of type
  {\eqref{wz3}}.
  
  Let $\xi \assign \tmop{ord}_{\mathbf{v}} F = \tmop{ord}_{\mathbf{v}} G$, $t
  \assign \tmop{ord}_{\mathbf{v}} \omega$, and let $F = \sum_{i \geqslant \xi}
  F^{(i)}$, $G = \sum_{i \geqslant \xi} G^{(i)}$, $A = \sum_{i \geqslant \delta}
  A^{(i)}$, $B = \sum_{i \geqslant \eta} B^{(i)}$ be the decompositions of the
  involved series into their quasihomogeneous components for the gradation
  defined by the weight vector $\mathbf{v}$. Thus, $F^{(\xi)} =
  F_{\mathbf{v}}$ and $G^{(\xi)} = G_{\mathbf{v}}$. Notice that
  $\tmop{ord}_{\mathbf{v}} A \geqslant \xi$ and $\tmop{ord}_{\mathbf{v}} B \geqslant
  \xi$, too. Indeed, the (non-zero) components of lowest degree of $AF$ and
  $BG$ are, respectively, equal to $A^{(\delta)} F^{(\xi)}$ and $B^{(\eta)}
  G^{(\xi)}$; clearly, $\min (\delta + \xi, \eta + \xi) \leqslant t$. This
  inequality must be strict because of our second assumption. But then, as
  $A^{(\delta)} \neq 0 \neq B^{(\eta)}$, we get $\delta = \eta$ and
  $A^{(\delta)} F^{(\xi)} + B^{(\eta)} G^{(\xi)} = 0$. Since $R [X_1, \ldots,
  X_n]$ is a {\tmname{UFD}} and $\gcd (F^{(\xi)}, G^{(\xi)}) = 1$, we infer
  that $F^{(\xi)} \divides B^{(\eta)}$ and $G^{(\xi)} \divides A^{(\delta)}$;
  hence, $\xi \leqslant \eta$, $\xi \leqslant \delta$, and $2 \xi < t$.
  
  Set $C^{(2 \xi + n)} \assign \sum_{0 \leqslant i \leqslant n} (A^{(\xi + i)} F^{(\xi +
  n - i)} + B^{(\xi + i)} G^{(\xi + n - i)})$, for $n \geqslant 0$. By the above,
  $C^{(2 \xi + n)}$ is the component of degree $(2 \xi + n)$ of the left-hand
  side of {\eqref{wz3}}. We shall construct a sequence $(P^{(j)})_{0 \leqslant j
  \leqslant t - 2 \xi - 1}$ of quasihomogeneous polynomials $P^{(j)} \in R [X_1,
  \ldots, X_n]$ of $\mathbf{v}$-degree $j$ such that
  
  \begin{equation}
    \left\{\begin{array}{lll}
      A^{(\xi + n)} & = & \sum_{0 \leqslant j \leqslant n} P^{(j)} G^{(\xi + n - j)}\\
      B^{(\xi + n)} & = & - \sum_{0 \leqslant j \leqslant n} P^{(j)} F^{(\xi + n - j)}
    \end{array}\right., \text{ for } 0 \leqslant n \leqslant t - 2 \xi - 1.
    \label{rekurencja}
  \end{equation}
  
  We have $2 \xi < t$ so, as above, we get $0 = C^{(2 \xi)} = A^{(\xi)}
  F^{(\xi)} + B^{(\xi)} G^{(\xi)}$. Thus, we may set $P^{(0)} \assign
  A^{(\xi)} / G^{(\xi)} = - B^{(\xi)} / F^{(\xi)} \in R [X_1, \ldots, X_n]$.
  Then {\eqref{rekurencja}} holds with $n = 0$.
  
  Say, we have already defined $(P^{(j)})_{0 \leqslant j \leqslant N - 1}$ for some $1
  \leqslant N \leqslant t - 2 \xi$. We have
  
  \begin{align*}
    C^{(2 \xi + N)} &= \sum_{0 \leqslant i \leqslant N} (A^{(\xi + i)} F^{(\xi + N - i)} + B^{(\xi +
    i)} G^{(\xi + N - i)}) = A^{(\xi + N)} F^{(\xi)} + B^{(\xi + N)} G^{(\xi)}+\\
    &+ \sum_{0 \leqslant i \leqslant N - 1}\,\, \sum_{0 \leqslant j \leqslant i} (P^{(j)} G^{(\xi + i
    - j)} F^{(\xi + N - i)} - P^{(j)} F^{(\xi + i - j)} G^{(\xi + N - i)}) =\\
    &= A^{(\xi + N)} F^{(\xi)} + B^{(\xi + N)} G^{(\xi)}+\\
	&+ \sum_{0 \leqslant j
    \leqslant N - 1} P^{(j)}  \sum_{j \leqslant i \leqslant N - 1} (G^{(\xi + i - j)} F^{(\xi
    + N - i)} - F^{(\xi + i - j)} G^{(\xi + N - i)}) .
  \end{align*}
  After interchanging the order of the summands, the inner sum becomes
  \[ \sum_{j \leqslant i \leqslant N - 1} G^{(\xi + i - j)} F^{(\xi + N - i)} - \sum_{j
     + 1 \leqslant i \leqslant N} F^{(\xi + N - i)} G^{(\xi + i - j)} = G^{(\xi)}
     F^{(\xi + N - j)} - F^{(\xi)} G^{(\xi + N - j)} . \]
  Hence,
  \begin{align*}
		C^{(2 \xi + N)} & = A^{(\xi + N)} F^{(\xi)} + B^{(\xi + N)} G^{(\xi)} +
    \sum_{0 \leqslant j \leqslant N - 1} P^{(j)}  (G^{(\xi)} F^{(\xi + N - j)} -
    F^{(\xi)} G^{(\xi + N - j)})=\\
    & = F^{(\xi)}  \left( A^{(\xi + N)} - \sum_{0 \leqslant j \leqslant N - 1}
    P^{(j)} G^{(\xi + N - j)} \right) + G^{(\xi)}  \left( B^{(\xi + N)} +
    \sum_{0 \leqslant j \leqslant N - 1} P^{(j)} F^{(\xi + N - j)} \right) .
  \end{align*}

  If $N \leqslant t - 2 \xi - 1$, then $C^{(2 \xi + N)} = 0$ and, as above, using
  the fact that $R [X_1, \ldots, X_n]$ is a {\tmname{UFD}} and $\gcd
  (F^{(\xi)}, G^{(\xi)}) = 1$, we define $P^{(N)}$ by $A^{(\xi + N)} - \sum_{0
  \leqslant j \leqslant N - 1} P^{(j)} G^{(\xi + N - j)} = P^{(N)} G^{(\xi)}$. But then,
  from the above equality we get $0 = P^{(N)} F^{(\xi)} + B^{(\xi + N)} +
  \sum_{0 \leqslant j \leqslant N - 1} P^{(j)} F^{(\xi + N - j)}$. Hence,
  {\eqref{rekurencja}} holds with $n = N$. Consequently, the sequence
  $(P^{(j)})_{0 \leqslant j \leqslant t - 2 \xi - 1}$ is defined by induction.
  
  If $N = t - 2 \xi$, then $C^{(2 \xi + N)} = \omega$. But then, we get
  $\omega \in (F^{(\xi)}, G^{(\xi)}) R [X_1, \ldots, X_n]$; contradiction.
						  
  \hspace*{\fill}\proofbox
\end{proof*}

Now we can give{\nopagebreak}

\begin{proof*}[of Theorem~{\rm\ref{tw_par}}]
  Clearly, we may assume that $\mathbf{v} = (v_x, v_y, v_z) \in
  \mathbb{N}_+^3$. We have $f_S = f_{\mathbf{v}}$. Since $S$ is a
  2-dimensional face and $\mathbf{v}$ has positive coordinates, $f_S$ depends
  in an essential way on $y$ and $z$, i.e.\ $\frac{\partial f_S}{\partial y}
  \neq 0$ and $\frac{\partial f_S}{\partial z} \neq 0$. Hence, $\frac{\partial
  f_S}{\partial y} = \frac{\partial f_{\mathbf{v}}}{\partial y} =
  (\frac{\partial f}{\partial y})_{\mathbf{v}}$ and $\frac{\partial
  f_S}{\partial z} = \frac{\partial f_{\mathbf{v}}}{\partial z} =
  (\frac{\partial f}{\partial z})_{\mathbf{v}}$. In view of the assumptions
  and the Bernstein theorem (Theorem \ref{ber}), the pair $\{ \frac{\partial
  \overline{f_S}}{\partial y} = 0, \frac{\partial \overline{f_S}}{\partial z}
  = 0\}$ possesses $\tmop{MV} (\frac{\partial \overline{f_S}}{\partial y},
  \frac{\partial \overline{f_S}}{\partial z}) > 0$ solutions in
  $(\mathbb{C}^{\ast})^2$. Hence, $\{ \frac{\partial f_S}{\partial y} = 0,
  \frac{\partial f_S}{\partial z} = 0\}$ possesses solutions in
  $(\mathbb{C}^{\ast})^3$. This means that in the ideal $((\frac{\partial
  f}{\partial y})_{\mathbf{v}}, (\frac{\partial f}{\partial z})_{\mathbf{v}})
  \mathcal{O}_3 = (\frac{\partial f_S}{\partial y}, \frac{\partial
  f_S}{\partial z}) \mathcal{O}_3$ there are no monomials.
  
  Moreover, applying Lemma \ref{le_degeneracy} to $F = f_S$, we infer that
  $\frac{\partial f_S}{\partial y}, \frac{\partial f_S}{\partial z}$ have no
  common factor in $\mathbb{C} [x, y, z]$ except a monomial.
  
  Summing up, we have checked that the pair $\frac{\partial f}{\partial y},
  \frac{\partial f}{\partial z}$ in $\mathcal{O}_3 =\mathbb{C} \{x, y, z\}$,
  and the weight vector $\mathbf{v}$ fulfil the assumptions of Lemma
  \ref{ret}. Therefore, $\frac{\partial f}{\partial y}$ and $\frac{\partial
  f}{\partial z}$ do not generate any element in $\mathcal{O}_3$ with initial
  $\mathbf{v}$-part being a monomial. This, in turn, means we are in the
  position to apply the Maurer theorem (Theorem \ref{mau}) to the curve $\{
  \frac{\partial f}{\partial y} = 0, \frac{\partial f}{\partial z} = 0\}$ and
  the vector $\mathbf{v}$. Thus, we are delivered the required parametrization
  $\Phi$.\hspace*{\fill}\proofbox
\end{proof*}

Although generic, the Bernstein non-degeneracy is a necessary assumption in
Theorem \ref{tw_par} (it does not follow from Kushnirenko non-degeneracy).
This is illustrated by the following{\nopagebreak}

\begin{example*}
  Let $f \assign (x^4 + y^3) z + x^4 y + \frac{1}{4} y^4 + x^6 + z^5$. Then
  $\Gamma^2 (f)$ is built of three faces, two of which are
  non-exceptional. Consider the face $S \assign \mathcal{N} \left( (x^4
  + y^3) z + x^4 y + \frac{1}{4} y^4 \right)$, a parallelogram in space.
  As its normal vector, we may take $\mathbf{v} \assign
  (3, 4, 4)$. It is easy to check that $f$ is $\mathcal{K}\!\!$-non-degenerate. We
  have $\left( \tfrac{\partial \overline{f_S}}{\partial y}, \frac{\partial
  \overline{f_S}}{\partial z} \right) = (1 + y^3 + 3 y^2 z, 1 + y^3)$. Hence,
  $\tmop{MV} \left( \tfrac{\partial \overline{f_S}}{\partial y},
  \frac{\partial \overline{f_S}}{\partial z} \right) = 3 > 0$. However,
  $\left\{ \tfrac{\partial \overline{f_S}}{\partial y} = \frac{\partial
  \overline{f_S}}{\partial z} = 0 \right\} \subset \{ z = 0 \}$ so the system
  has no solutions in $(\mathbb{C}^{\ast})^2$. A fortiori, there are no
  parametrizations of $\left\{ \tfrac{\partial f}{\partial y} = \frac{\partial
  f}{\partial z} = 0 \right\}$ with initial exponent vector proportional to
  $\mathbf{v}$, contrary to the assertion of Theorem \ref{tw_par}. The reason
  behind this is that the system $\left\{ \tfrac{\partial
  \overline{f_S}}{\partial y} = \frac{\partial \overline{f_S}}{\partial z} = 0
  \right\}$ is $\mathcal{B}$-degenerate: for the vector $\mathbf{w} \assign
  (3, 4) \in \mathbb{N}^2$ we have $\left( \tfrac{\partial
  \overline{f_S}}{\partial y} \right)_{\mathbf{w}} = \left( \frac{\partial
  \overline{f_S}}{\partial z} \right)_{\mathbf{w}} = x^4 + y^3$, proving the
  degeneracy. Note also that $S$ is a proximate face for the
  axis $0 \nospace x$ (see Definition \ref{de_prox} below). 
\end{example*}

It turns out that under some further restrictions, including ones on the geometry of the face, 
we may avoid the misbehaviour from the above example. Namely, we have.

\begin{theorem}\label{prop}Let $f : (\mathbb{C}^3, 0) \rightarrow (\mathbb{C}, 0)$ be an
  isolated singularity, and let $S \in \Gamma^2 (f)$ satisfy $\tmop{MV}
  (\frac{\partial \overline{f_S}}{\partial y}, \frac{\partial
  \overline{f_S}}{\partial z}) > 0$. Assume that no edges of $S$ have
  prolongations intersecting the axis $0 \nospace x$,
  % except those lying in
  % the coordinate planes $0 \nospace x \nospace y$, $0 \nospace x \nospace z$.
  {except, perhaps, those having a vertex in $0 \nospace x \nospace y \cup 0 \nospace x \nospace z$.}
  Assume, moreover, that $f_S$ has generic coefficients and $\tmop{supp} f_S =
  \Gamma^0 (f_S)$. Then there exists a parametrization $0 \neq \Phi =
  (\varphi_1, \varphi_2, \varphi_3) \in \mathbb{C} \{t\}^3$ of a branch of the
  curve $\{ \frac{\partial f}{\partial y} = 0, \frac{\partial f}{\partial z} =
  0\}$ at $0$, whose initial exponent vector is perpendicular to $S$.
\end{theorem}

The above corollary follows directly from Theorem \ref{tw_par}, genericness of
$\mathcal{K}\!\!$-non-degeneracy (see
{\cite{Kou76}}, {\cite{Oka79}}), and the following observation.

\begin{lemma}
  \label{lem_pom}Under the above assumptions, the pair $(\frac{\partial
  \overline{f_S}}{\partial y}, \frac{\partial \overline{f_S}}{\partial z})$ is
  $\mathcal{B}$-non-degenerate.
\end{lemma}

\begin{proof*} Take any $\mathbf{w} = (w_1, w_2) \in \mathbb{Z}^2
\setminus \{0\}$ and consider the system of equations
\begin{equation}
  \left( \frac{\partial \overline{f_S}}{\partial \nospace y}
  \right)_{\mathbf{w}} = \left( \frac{\partial \overline{f_S}}{\partial
  \nospace z} \right)_{\mathbf{w}} = 0. \label{13}
\end{equation}
We must show that this system has got no solutions in $(\mathbb{C}^{\ast})^2$.
Since $\tmop{supp}f_S$ is equal to the set of vertices of $S$, the polynomials
$(\frac{\partial \nospace \overline{f_S}}{\partial y})_{\mathbf{w}}$ and
$(\frac{\partial \nospace \overline{f_S}}{\partial z})_{\mathbf{w}}$ are
monomials or binomials only. Clearly, it suffices to consider the case both of
them are binomials. Then their supports define parallel segments. If these
binomials are not partial derivatives of one and the same edge (or diagonal)
of $\overline{f_S}$, then, because of genericness of coefficients of $f_S$ and
the parallelity of the segments, system {\eqref{13}} has got no solutions in
$(\mathbb{C}^{\ast})^2$.\label{ewentualnie tutaj bernstein} In the opposite
case, we have $(\frac{\partial \overline{f_S}}{\partial y})_{\mathbf{w}} =
\frac{\partial \overline{f_T}}{\partial y}$, $(\frac{\partial
\overline{f_S}}{\partial z})_{\mathbf{w}} = \frac{\partial
\overline{f_T}}{\partial z}$, where $T$ is an edge or a diagonal of $S$ (the
latter may happen if some $x^i \in \tmop{supp} f_S$). The polynomial $f_T$ is
also a binomial.
% As $\frac{\partial f_T}{\partial y} \neq 0$ and
% $\frac{\partial f_T}{\partial z} \neq 0$, the segment $T$ is not contained in
% any of the coordinate planes $0 \nospace x \nospace y$, $0 \nospace x \nospace
% z$.
{As $\frac{\partial f_T}{\partial y}$ and $\frac{\partial f_T}{\partial z}$
are binomials, the segment $T$ is disjoint from both coordinate planes 
$0 \nospace x \nospace y$, $0 \nospace x \nospace z$.}
If $T \in \Gamma^1 (f_S)$, then, directly by the assumption, we get that
the line containing $T$ does not intersect the axis $0 \nospace x$. If this is
not the case, viz., $T$ is a diagonal of $S$, the only possibility is that $T$
connects the two vertices of $\mathcal{N} (f_S)$ joined by edges with $x^i$
(as $\tmop{supp} f_S = \Gamma^0 (f_S)$). Thus, also in this situation, the
line containing $T$ does not intersect the axis $0 \nospace x$.

Clearly, $\overline{f_T}$ takes one of the following two forms: 1) $y^l z^m 
(\alpha y^a + \beta z^b)$, with $l, m, a, b > 0$, or 2) $y^l z^m  (\alpha +
\beta y^a z^b)$, with $l, m > 0, a + b > 0$, where the coefficients $\alpha,
\beta$ are generic. An easy analysis leads to the conclusion that the system
$\{ \frac{\partial \overline{f_T}}{\partial y} = 0, \frac{\partial
\overline{f_T}}{\partial z} = 0\}$ has got a solution in
$(\mathbb{C}^{\ast})^2$ if and only if $\overline{f_T}$ is of the second form,
where, moreover, $a, b > 0$ and $\frac{k}{l} = \frac{a}{b}$. This last
condition, however, means the prolongation of $T$ intersects the axis $0
\nospace x$; impossible.
\hspace*{\fill}\proofbox
\end{proof*}

\section{The geometry of proximity faces}\label{geometry}

%Observe
%(see page \pageref{uzas_prox})
Let $g:(\mathbb{C}^3, 0) \rightarrow (\mathbb{C},0)$ 
be an isolated singularity and $\omega \in\{x,y,z\}.$
If $\Gamma^2(g)\setminus E_{\omega}(g) \neq \varnothing$,
% {Let $g : (\mathbb{C}^3, 0) \rightarrow (\mathbb{C},
% 0)$ be an isolated singularity with $\Gamma^2 (g) \neq \varnothing$}. For every axis $0 \omega$, where $\omega \in \{x, y, z\}$,
the maximal $m (S)_\omega$
among $S\in\Gamma^2 (g)\setminus E_\omega (g)$
is always achieved on an $0 \omega$-non-exceptional face
%of $\Gamma (g)$
of a special type, a face -- in a sense -- ``closest'' to one of the
coordinate axes (see Lemma \ref{le_geo} below). Moreover, then, there are certain restrictions on the
geometry of such a face. This leads to the following definition.

\begin{definition}
  \label{de_prox} Let $\omega \in \{x, y, z\}$. We say that
$S \in \Gamma^2 (g)$ is {\tmem{proximate for the axis $\nospace 0 \nospace
\omega$}} if $S$ is non-exceptional with respect to the axis $0 \nospace
\omega$, has a vertex lying at distance $\leqslant 1$ from this axis, and
touches both coordinate planes containing this axis. A {\tmem{proximity
face}} is one which is proximate for one of the axes.
\end{definition}

In particular, according to the above definition, a face that is $0 \omega$-non-exceptional
and has a vertex on the axis $0 \omega$ is $0 \omega$-proximate.
The existence of a proximity face for some axis is characterized by the two following propositions.
{They are simple consequences of Lemma 3.1 and Theorem 3.8 in \cite{Ole13}.}

\begin{proposition}\label{exist_prox}%Let $g:(\mathbb{C}^3,0)\rightarrow(\mathbb{C},0)$ be an isolated singularity and 
Let $\omega \in\{x,y,z\}.$ There exists a proximity face for\/ $0\omega$ if, and only if,\/ $\Gamma^2(g)\setminus E_{\omega}(g)\neq \varnothing.$ Moreover, if\/ $\Gamma^2(g)\setminus E(g)\neq \varnothing$, then all the proximity faces are
non-exceptional.
\end{proposition}

% \begin{proposition}\label{characterization}%Let $g:(\mathbb{C}^3,0)\rightarrow(\mathbb{C},0)$ be an isolated singularity. Then\/ 
% $\Gamma^2(g) \setminus E_x(g)=\varnothing$ if, and only if, there exists an edge $\mathcal{N}(x^{}y+z^{\alpha})\in\Gamma^1(g)$, for
% some ${\alpha}\geqslant 2$ (up to permutation of {$(x,y,z)$}), and $E(g)=E_x(g).$
% \end{proposition}

\begin{proposition}\label{characterization}
{Let $\omega \in \{ x, y, z \}$. The following conditions are equivalent:
\begin{enumerate}
  \item $\Gamma^2 (g) \setminus E_{\omega} (g) = \varnothing$.
  
  \item $E (g) = E_{\omega} (g)$ and, up to permutation of the variables $(x,
  y, z)$, there exists an edge $\mathcal{N} (x {} y + z^{\alpha}) \in
  \Gamma^1 (g)$, for some $\alpha \geqslant 2$.
\end{enumerate}
}
\end{proposition}

It is easy to observe that a face can be exceptional with respect to one of the axes only, i.e.
 $E_x(g)\cap E_y(g)=\varnothing$, $E_x(g)\cap E_z(g)=\varnothing$, $E_y(g)\cap E_z(g)=\varnothing$.
 Hence, assuming {$\Gamma^2 (g) \neq \varnothing$}, by Proposition \ref{exist_prox} there exist proximity faces for $g$ for at least two coordinate axes. But not necessarily for all three axes, as the following example shows.

{\begin{example}\label{rem_prox}
In a typical situation, a proximity face $S$ is
non-exceptional; in general, however, this is not the case, e.g.~$S \assign
\mathcal{N} (g)$, where $g \assign xz + yz + y^3$, is $0 \nospace x$-proximate
but $\frac{\partial g}{\partial x} = z$ so $S$ is exceptional with respect to
the axis $0 \nospace z$. This also means that the isolated singularity $g$ does not 
have any proximity face for the axis $0 \nospace z$. Observe that, by Proposition \ref{exist_prox}, this phenomenon can only happen
if $\Gamma^2 (g) \setminus E (g)=\varnothing$.
\end{example}}

%If one assumes, however, that $\Gamma^2 (g) \setminus E (g) \neq \varnothing$, then -- again by Proposition \ref{exist_prox} -- there exist
%non-exceptional proximity faces for all the three axes. %, and these proximity faces are non-exceptional
%(see \cite[Lemma 3.1]{Ole13}).
\medskip

 Let 
$S\in\Gamma^2 (g)$ be proximate for, say, the axis $0 \nospace x$. One can
observe that two situations may occur (cf.\ \cite[Lemma 3.1 and Property 3.3]{Ole13}):
\begin{enumeratealpha}
  \item \label{it_dog}$S$ is
  {\tmem{convenient with respect to the axis $0 \nospace x$}} (i.e.\ $S$
  touches $0 \nospace x$) and then this proximity face is not necessarily
  unique for $0 \nospace x$ -- see Figure \ref{fig_proximity} (a). In this
  case, all the faces proximate for the axis $0 \nospace x$ share one common
  vertex, $x^{m (S)_x}$.\label{To wynika z braku wyj�tkowych}
  
  \item \label{it_niedog}$S$ is {\tmem{non-convenient with respect to the axis
  $0 \nospace x$}} (i.e.\ $S$ is disjoint from $0 \nospace x$); then it is
  unique and has an edge joining some monomials of the form $x^k z$, $k \geqslant
  1$, and $x^m y^n$, $m \geqslant 0, n \geqslant 1$ (or similar monomials with the
  variables $y$ and $z$ permuted). Moreover, in $\tmop{supp} g$ restricted to
  the plane $0 \nospace x \nospace y$, the point $x^m y^n$ has the smallest
  possible $y$-coordinate among all $0 \nospace x$-non-exceptional faces. See Figure \ref{fig_proximity} (b).
\end{enumeratealpha}
\begin{figure}[h!]
  \resizebox{343pt}{101pt}{\includegraphics{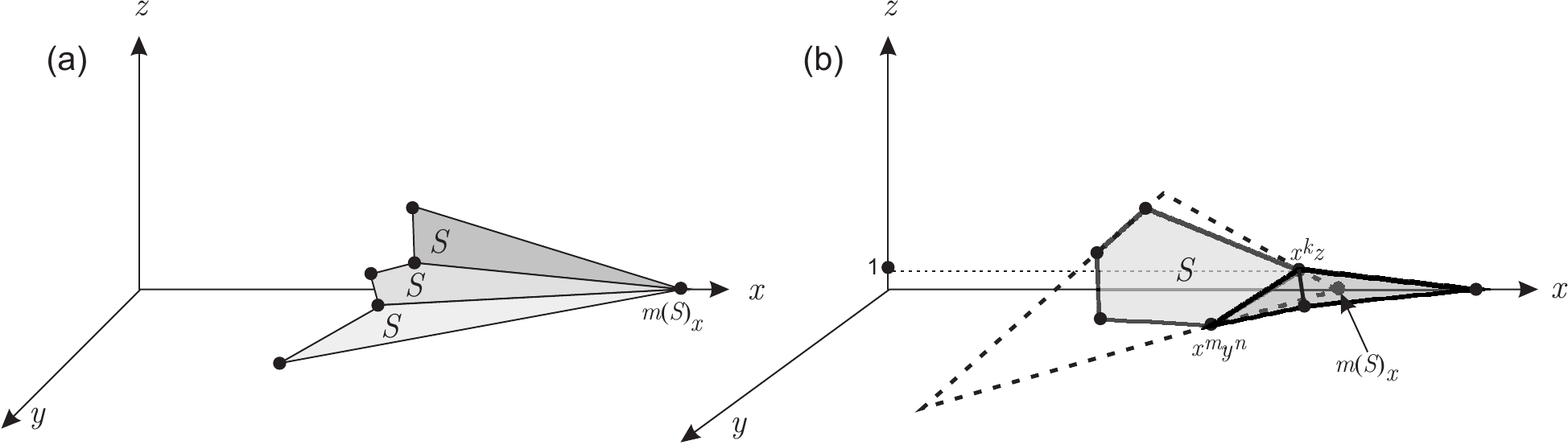}}
  \caption{\label{fig_proximity}Proximity faces $S$ for the axis $0 
\nospace x$.}
\end{figure}

\smallskip
Note that, in both above situations, for a vector $\mathbf{v}_S = (v_x, v_y,
v_z) \in \mathbb{N}_+^3$ normal to $S$, we have
%$\min (v_x, v_y, v_z) = v_x$ and 
$m (S) \geqslant m (S)_x = \frac{l_S}{v_x}$, where $l_S$ is the
$\mathbf{v}_S$-degree of $g_S$ (also $\mathbf{v}_S$-order of $g$).

\medskip
%We omit elementary proofs of the following two facts on an $0 \nospace x$-proximate face $S$:
We omit elementary proofs of the following three facts on proximity faces (stated for the axis $0 x$ for simplicity;
analogous statements hold for the other axes):

\begin{lemma}\label{unique_convenient} Let $S$ be a proximate face for the axis $0x$, convenient with respect to this axis.  If $S$ has a diagonal joining some monomials of the form $x^k z$, $k \geqslant
  1$, and $x^m y^n$, $m \geqslant 0, n \geqslant 1$, then $S$ is the unique proximity face for the axis $0x$.
\end{lemma}

\begin{lemma}\label{le_geo}
% %Assume ({\ref{ass_3}}) and $m(f)=m(f)_x$. Then $m(f)=m(S)=m(S)_x$.
% Among all the non-exceptional faces, the face $S$ is the one whose supporting
% plane has the highest coordinate of intersection with the axis $0 \nospace x$,
% i.e.
% \[ m (S)_x = \max_{T \in \Gamma^2 (f) \setminus E (f)} m (T)_x . \]
Let $S$ be a proximate face for the axis $0\nospace x$.
%, where $\omega \in \{x, y, z\}$.
%{Assume that $S$ is non-exceptional}.
Then the supporting plane of $S$ has the highest coordinate of intersection
with the axis $0\nospace x$ among all the $0\nospace x$-non-exceptional faces.
%An analogous statement holds if $S$ proximate for the other axes.
%, i.e.
%\[ \max_{T \in \Gamma^2 (g) \setminus E (g)} m (T)_x = m (S)_x. \]
\end{lemma}

\begin{lemma}\label{proper_edge}
%If $S$ is proximate for $0 \nospace x$, then
% There is no edge of $S$ disjoint from both coordinate planes $0 \nospace x
% \nospace y$, $0 \nospace x \nospace z$ whose pro\-long\-a\-tion has a common point
% with the axis $0 \nospace x$.
Let $S$ be a proximate face for the axis $0x$. %(respectively
%$0y,0z$). 
Then no edges of $S$ have prolongations intersecting the axis $0 \nospace x$,
{except, perhaps, those having at least one vertex in one of the planes $0 \nospace x \nospace y$, $0 \nospace x
\nospace z$.}
% except those lying in the planes $0 \nospace x \nospace y$ and $0 \nospace x
% \nospace z$.
%An analogous statement holds if $S$ proximate for the other axes.

% {Lub: (...) except perhaps those having a vertex in $0 \nospace x \nospace y \cup 0 \nospace x \nospace z$.}
\end{lemma}

% \begin{proof}
  % Indeed, by its definition, a proximate face touches both planes $0 \nospace
  % x \nospace y$, $0 \nospace x \nospace z$; hence, it has vertices in these
  % planes. Moreover, $S \subset \mathbb{R}_+^3$ and $S$ is convex. This means
  % that the only edges of $S$ that may get prolonged to meet the point $(m
  % (S)_x, 0, 0)$ on the axis $0 \nospace x$ necessarily touch the
  % aforementioned planes.
% \end{proof}

\section{Existence of parametrizations for a proximity face}

Here, we indicate a way to extend the main result of section \ref{Wide-kreska} also to the case of ``zero mixed volume'',
but only for proximity faces (as their geometry is ``good'', in the sense of Theorem \ref{prop} above). 
Although one cannot claim now that there always exists a parametrization of the type postulated in Theorem \ref{tw_par}, it turns out that, nevertheless, there exists a parametrization of a branch of the polar curve $\{ \frac{\partial
  f}{\partial y} = 0, \frac{\partial f}{\partial z} = 0\}$ whose initial exponent vector
 supports a certain, rather special, subface of the chosen $0 x$-proximate face. Precisely, we have the following

\begin{theorem}\label{key}Let $f : (\mathbb{C}^3, 0) \rightarrow (\mathbb{C}, 0)$ be an isolated singularity and $S$ be its proximity face for $0x.$
  Assume, moreover, that $\tmop{supp} f=\Gamma^0 (f)$ and $f$ has generic coefficients. Then there exists a parametrization $\Phi =
  (\varphi_1, \varphi_2, \varphi_3) \in \mathbb{C} \{t\}^3$ with $\Phi(0)=0$ and $\varphi_1 \neq 0$, such that 
\begin{equation}\label{zero}\frac{\partial f}{\partial y}\circ\Phi=\frac{\partial f}{\partial z}\circ\Phi = 0
\end{equation}
and its initial exponent vector is supporting either along $S$
or along an edge of $S$ lying in one of the planes $0x \nospace y$, $0x \nospace z$, or along
the vertex of $S$ on the axis $0x$.
\end{theorem}

\begin{remark*}Naturally, if $\varphi_i=0$, then the $i$-th coordinate of the initial exponent vector $\mathbf{v}$ of $\Phi$ is understood to be equal to $+\infty$. For convenience, we also define $l_\mathbf{v}$, $S_\mathbf{v}$, $\Pi_\mathbf{v}$
for such vectors, in a natural way, using the usual conventions for calculations involving infinity.
{
Equivalently, one may replace the infinite coordinates of $\mathbf{v}$ by a big enough $N \in \mathbb{N}$ to get the vector 
$\mathbf{\tilde{v}} \in \mathbb{Q}^3$; then, $l_\mathbf{v}=l_\mathbf{\tilde{v}}$, $S_\mathbf{v}=S_\mathbf{\tilde{v}}$, $\Pi_
\mathbf{v}=\Pi_\mathbf{\tilde{v}}$.}

\end{remark*}

\begin{proof} We split the proof into two main parts, cases to be treated separately.

\paragraph{\tmstrong{Part I.}}\label{pI}{\tmstrong{$S$ is $0x$-non-convenient.}} We consider two subcases
\paragraph{\tmtextbf{a)}}\label{re_I}{\tmstrong{$\tmop{MV} (\frac{\partial\overline{f_S}}{\partial y}, \frac{\partial
\overline{f_S}}{\partial z}) > 0$}}. \nopagebreak \hspace{1em} Here, by Lemma \ref{proper_edge}, we may apply Theorem \ref{prop} to find a parametrization $0 \neq \Phi \in \mathbb{C} \{t\}^3$ satisfying (\ref{zero}), whose initial exponent vector is perpendicular to $S$ so it supports $\Gamma_+(f)$ along $S$. 

\paragraph{\tmtextbf{b)}}{\tmstrong{$\tmop{MV} (\frac{\partial\overline{f_S}}{\partial y}, \frac{\partial \overline{f_S}}{\partial z}) = 0$}}.{\nopagebreak} \hspace{1em} In this situation, by Lemma \ref{lem}, $\mathcal{N} (\frac{\partial
\overline{f_S}}{\partial y})$ and $\mathcal{N} (\frac{\partial
\overline{f_S}}{\partial z})$ are either parallel segments or at least one of
these polygons reduces to a point. This implies certain restrictions on the
geometry of $S$. First, since a derivative can zero at most two vertices of
$S$, such $S$ cannot have more than four vertices. But ``four vertices''
implies that in the derivatives $\frac{\partial f_S}{\partial y}$,
$\frac{\partial f_S}{\partial z}$ there are segments at the least; and if this
happens, these segments either lie in different coordinate planes (if $S$ is
non-convenient for $0 \nospace x$) or they are derived from a common vertex of
$S$ (if $S$ is convenient for $0 \nospace x$); in either case, neither they
nor their projections can be parallel. This means $S$ has to be a
triangle.%{\medskip}

\par Since $S$ is a non-convenient proximity face, the
triangle $S$ has an edge joining the monomials $x^k{}z$, where $k \geqslant 1$, and
$x^m y^n$, where $m \geqslant 0$, $n \geqslant 1$ (up to permutation of the variables $y$ and $z$
-- recall Figure \ref{fig_proximity} (b)).

We claim that the last vertex of $S$ lies in the plane $0 \nospace x
\nospace z$. Indeed, it cannot lie in $0 \nospace x \nospace y$ for $S$ is
{$0x$-non-exceptional}; and if it were situated outside of both planes $0 \nospace x
\nospace y$ and $0 \nospace x \nospace z$, then we would have $\tmop{MV}
(\frac{\partial \overline{f_S}}{\partial y}, \frac{\partial
\overline{f_S}}{\partial z}) > 0$. Thus, $S =\mathcal{N} (x^m y^n + x^k z +
x^p z^q)$. As $S$ is {$0x$-non-exceptional}, $n \geqslant 2$. Moreover, $q \geqslant 1$ since
$S$ is non-convenient with respect to $0 \nospace x$. Hence, we get $p < k$
and $q \geqslant 2$.

We may, thus, write $\frac{\partial \nospace f_S}{\partial \nospace y} = ax^m
y^{n - 1}$ and $\frac{\partial f_S}{\partial z} = bx^k + cx^p z^{q - 1}$,
where $abc \neq 0$, and
\[ \frac{\partial f}{\partial z} (x, 0, z) = bx^k + cx^p z^{q - 1} + \cdots .
\]
We will show that there exists the sought $\Phi$ with the property that its initial exponent vector
supports $\Gamma_+(f)$
along the edge of $S$ lying in plane $0x \nospace z$.
To this end, let $\mathbf{w} = (w_x, w_y, w_z) \in \mathbb{N}_+^3$ be a supporting vector
of $\Gamma_+(f)$ along the edge of $S$ joining $x^k z$ and $x^p z^q$.
%(i.e.\ $\tmop{supp} f_\mathbf{w}$ consists only of the monomials $x^k z$  and $x^p z^q$).
To find such a vector $\mathbf{w}$, it is enough to start from a vector perpendicular to $S$ and then
increase its second coordinate (i.e.\ the one corresponding to $y$). If this increase is small enough,
then we may additionally assume that
%This way, for such a vector $\mathbf{w}$, it holds $m (S)_x = \frac{l_S}{w_x}$, where $l_S$ is the
%$\mathbf{w}$-order of $f$.
%Clearly, we may choose it so
%that $w_x = \min \{ w_x, w_y, w_z \}$.
$(\frac{\partial f}{\partial
y})_{\mathbf{w}} = a x^m y^{n-1}$.
 Thus,
 \[\left (\frac{\partial f}{\partial
y} \right )_{\mathbf{w}} = a x^m y^{n-1}, \quad \left ( \frac{\partial f}{\partial
z}\right )_{\mathbf{w}} = bx^k + cx^p z^{q - 1}.\]

%Consider $\frac{\partial f}{\partial y} (x, 0, z) $.
Letting $w_y \rightarrow \infty$
(and keeping $w_x$ and $w_z$ fixed), we may arrive at two cases.			 

\medskip First, if $\frac{\partial f}{\partial y} (x, 0, z) \nequiv 0$, then, since $n
\geqslant 2$, there exists a weight vector $\tilde{\mathbf{w}} = (w_x,
\tilde{w}_y, w_z)$ with $\tilde{w}_y > w_y$, for which $(\frac{\partial
f}{\partial y})_{\tilde{\mathbf{w}}} = ax^m y^{n - 1} + dx^r y^s z^t +
\left( \text{{\tmem{other possible terms, of }}} y \text{{\tmem{-degree less
than }}} n - 1 \right)$, where $d \neq 0$ and $s < n - 1$. This restriction on
the supported terms follows from the fact that any monomial from $\tmop{supp}
(\frac{\partial f}{\partial y})$ with $y$-degree $\geqslant n - 1$ is of weighted
$\tilde{\mathbf{w}}$-degree higher than the
$\tilde{\mathbf{w}}$-degree of $x^m y^{n - 1}$, just as was the case with
the $\mathbf{w}$-degree. On the other hand, we still have $(\frac{\partial
f}{\partial z})_{\tilde{\mathbf{w}}} = bx^k + cx^p z^{q - 1}$.
Consequently, as $s \neq n - 1$, if $\mathcal{N} ((\frac{\partial f}{\partial
y})_{\tilde{\mathbf{w}}})$ happens to be just the segment $\mathcal{N}
(x^m y^{n - 1} + x^r y^s z^t)$, then the projected segments $\mathcal{N}
\left( \overline{(\frac{\partial f}{\partial y})_{\tilde{\mathbf{w}}}}
\right)$ and $\mathcal{N} \left( \overline{(\frac{\partial f}{\partial
z})_{\tilde{\mathbf{w}}}} \right)$ are not parallel and, hence, have
positive mixed volume. Clearly, the same holds if $\mathcal{N}
((\frac{\partial f}{\partial y})_{\tilde{\mathbf{w}}})$ is a polygon.
Moreover, the pair of polynomials $\left( \overline{(\frac{\partial
f}{\partial y})_{\tilde{\mathbf{w}}}}, \overline{(\frac{\partial
f}{\partial z})_{\tilde{\mathbf{w}}}} \right)$ is
$\mathcal{B}$-non-degenerate because their coefficients originate from
{different vertices} of $\Gamma (f)$ and, as such, may be
considered generic.

By Bernstein theorem, the system $\left\{ \overline{(\frac{\partial
f}{\partial y})_{\tilde{\mathbf{w}}}} = \overline{(\frac{\partial
f}{\partial z})_{\tilde{\mathbf{w}}}} = 0 \right\}$ possesses $\tmop{MV}
\left( \overline{(\frac{\partial f}{\partial y})_{\tilde{\mathbf{w}}}},
\overline{(\frac{\partial f}{\partial z})_{\tilde{\mathbf{w}}}} \right) >
0$ solutions in $(\mathbb{C}^{\ast})^2$. Hence, $\left\{ (\frac{\partial
f}{\partial y})_{\tilde{\mathbf{w}}} = (\frac{\partial f}{\partial
z})_{\tilde{\mathbf{w}}} = 0 \right\}$ possesses solutions in
$(\mathbb{C}^{\ast})^3$. This means that in the ideal $((\frac{\partial
f}{\partial y})_{\tilde{\mathbf{w}}}, (\frac{\partial f}{\partial
z})_{\tilde{\mathbf{w}}}) \mathcal{O}_3$ there are no monomials. What is
more, using -- again -- the relative genericness of coefficients of the
polynomials $(\frac{\partial f}{\partial y})_{\tilde{\mathbf{w}}}$,
$(\frac{\partial f}{\partial z})_{\tilde{\mathbf{w}}}$, we infer that
these polynomials do not have any common factor except, possibly, a monomial.
Consequently, Lemma \ref{ret} implies there is no element $H \in
(\frac{\partial f}{\partial y}, \frac{\partial f}{\partial z}) \mathcal{O}_3$
such that $H_{\tilde{\mathbf{w}}}$ is a monomial. Thus, we may
apply Maurer theorem (Theorem \ref{mau}) to find a parametrization $\Phi (t) =
(\varphi_1 (t), \varphi_2 (t), \varphi_3 (t))$ of the curve $\{ \frac{\partial
f^{^{^{}}}}{\partial y} = \frac{\partial f}{\partial z} = 0\}$ whose initial exponent vector
is a multiple of $\tilde{\mathbf{w}}$ (hence, $\varphi_1 \neq 0$).
But, by its construction, $\tilde{\mathbf{w}}$ supports $\Gamma_+(f)$
along the edge of $S$ lying in plane $0xz$.
% supporting for the
% edge of $S$ lying in plane $0x \nospace z$.
% From the choice of
% $\tilde{\mathbf{w}}$, we infer that $\tmop{ord} \Phi = \tmop{ord}
% \varphi_1$. \label{a co z niezerowo�ci� - niedegeneracja}

\medskip Second, if $\frac{\partial f}{\partial y} (x, 0, z) \equiv 0$, then let $\Psi
(t) = (\psi_1 (t), \psi_2 (t))$ be a parametrization of a branch of the curve
$\{ \frac{\partial \nospace f}{\partial \nospace z} (x, 0, z) = 0\}$
corresponding to the edge $T \assign \mathcal{N} (x^k + x^p z^{q - 1})$ of its
Newton polyhedron $\Gamma_+ (\frac{\partial \nospace f}{\partial \nospace z} (x, 0,
z))$ in $\mathbb{C}^2$. Put $\Phi (t) = (\psi_1 (t), 0, \psi_2 (t))$. Then $\psi_1(t) \neq 0$ and 
and the initial exponent vector of $\Phi$ supports $\Gamma_+(f)$
along the edge of $S$ lying in plane $0xz$.
%Moreover, if required, we may easily modify $\Phi$ to assume it
%has non-zero coordinates and supports the edge $\mathcal{N} (x^k z + x^p z^q)$
%in $\Gamma (f)$.
This ends the proof of subcase \tmstrong{b)}. 
\medskip

\paragraph{\tmstrong{Part II.}}\label{P2}{\tmstrong{$S$ is $0x$-convenient.}} In this case, we search for $\Phi$
whose initial exponent vector supports $\Gamma_+(f)$						   
along a subface of $S$ having $x^k$ as a vertex.
We will analyze the behaviour of the series $f^{\ast} \assign f - e x^k$, $e \neq 0$, that arises from $f$ by removing the monomial $x^k$ from its support.

First, let $f^{\ast}$ be non-isolated. If $f^{\ast}$ is not nearly convenient with respect to $0x$, then we put $\Phi = (t,0,0)$ and its initial exponent vector $(1,+\infty,+\infty)$ supports $\Gamma_+(f)$ along the vertex $x^k$.
 Otherwise, there exists a vertex realizing the nearly convenience, say $x^lz$, where $l \geqslant 1$.
% Then if $\Gamma_+(f^{\ast})\cap 0xy \neq \varnothing,$ by Proposition \ref{BrOl} $f^{\ast}$ is an isolated singularity, contradiction.
But then, from Proposition \ref{BrOl}, we infer that necessarily $\Gamma_+(f^{\ast})\cap 0xy = \varnothing$.
%Hence $\Gamma_+(f^{\ast})\cap 0xy = \varnothing.$
%In this case we indicate a suitable $\Phi.$ Put $\Phi=(\varphi_1,\varphi_2,0).$
We shall indicate the suitable $\Phi$ in the form $\Phi=(\varphi_1,\varphi_2,0)$.
Clearly, $\varphi_3=0$ already gives $\frac{\partial f}{\partial y}\circ\Phi = 0$.
 Now, we shall choose $(\varphi_1,\varphi_2)$ for which $\frac{\partial f}{\partial z}\circ\Phi = 0$.
 From the above analysis, it follows that we may write $f(x,y,z)=zg(x,y) + z^2h(x,y,z) + e x^k$,
 where $g(0,0)=0$.
 Thus, $\frac{\partial f}{\partial z}\circ\Phi = g(\varphi_1,\varphi_2)$. But $g(x,y) =\theta_1 x^l +\theta_2 y^m +...$,
 where $\theta_1, \theta_2 \neq 0$, $m \geqslant 1$, because $f$ is nearly convenient with respect to the axis $0y$.
 By the Newton-Puiseux theorem, we may find $(\varphi_1,\varphi_2)$ with $\varphi_1 \neq 0$, 
 $\varphi_1(0)=\varphi_2(0)=0$ and $g(\varphi_1,\varphi_2)=0$.
 Then $\frac{\partial f}{\partial z}\circ\Phi = 0$ and
 the initial exponent vector $(\rm{ord}\, \varphi_1,\rm{ord}\, \varphi_2,+\infty)$
 of $\Phi$ supports $\Gamma_+(f)$ along the vertex $x^k$.

Now, let $f^{\ast}$ be an isolated singularity.
%This subcase  follows from part \tmstrong{I} of the proof.
%Indeed, if $\Gamma^2 (f^{\ast})\setminus E_x(f^{\ast}) = \varnothing$,
Suppose that $\Gamma^2 (f^{\ast})\setminus E_x(f^{\ast}) = \varnothing$.
Then, using Proposition \ref{characterization}, we infer that $f^{\ast}$ has an edge
$\mathcal{N}(x^{}y + z^{\alpha})$ (up to permutation of $(y,z)$ only),
for some $\alpha \geqslant 2$, and $E_x(f^{\ast})=E(f^{\ast})$. Hence, this edge is also an edge of a non-compact, parallel to $0y$, face of $f^{\ast}$. 
Therefore, it is an edge of $f$, too, and $\varnothing \neq E_x(f)=E(f)$. Then, again by Proposition \ref{characterization}, we get there is no $0x$-non-exceptional face of $f$, a contradiction.
Hence, $\Gamma^2 (f^{\ast})\setminus E_x(f^{\ast}) \neq \varnothing$. By Proposition \ref{exist_prox}, there exists an $0x$-proximate face $T$ of $f^{\ast}$, non-convenient with respect to this axis.
Clearly, we are in the position to apply part \tmstrong{I} of the proof to $f^{\ast}$ and $T$. This way, we find a
parametrization satisfying \eqref{zero} such that its initial exponent vector $\mathbf{v}$
{supports $\Gamma_+(f^{\ast})$ along $N$, where either $N=T$ or $N$ is an edge of $T$ lying in one of the planes $0xy$, $0xz$.}
 Take the supporting plane $\Pi=\Pi_\mathbf{v}$ of $\Gamma_+ (f^{\ast})$ {along $N$}. If $x^k$ lies above $\Pi$, then $x^k$ also lies above the
 plane spanned by $T$ 
 %(since, in any case, the affine hull of $N$ intersects $0 x$ in a point and $N \subset T$);
 (because the affine hulls of $T$ and $N$ both intersect $0 x$ in the same single point);
 thus, $T$ happens to be an $0 x$-proximate face of $f$.
 But then, by Lemma \ref{unique_convenient}, $T$ is the unique $0 x$-proximate face of $f$ so $S=T$. 
 As $S$ is $0x$-convenient while $T$ is not, this is a contradiction. 
 % contradiction (see \ref{it_dog})
 % in section \ref{geometry}).
 
 Hence, $x^k$ lies on $\Pi$ or under it. If $x^k$ lies on $\Pi$ and $N=T$, then, by Lemma \ref{unique_convenient}, we get that $T$ and $x^k$ make the unique {$0 x$-proximate} face of $f$. 
 Therefore, $\tmop{conv} (T \cup x^k) = S$ and, by part \tmstrong{I} of the proof, the vector $\mathbf{v}$ supports $\Gamma_+(f)$ along $S$ in this case.
 If $x^k$ lies on $\Pi$ and $N$ were an edge of $T$ lying in one of the planes $0xy$ or $0xz$, then $N$ and $x^k$ would make an edge of $f$
 containing three points of $\tmop{supp} f$; contradiction with the assumption $\tmop{supp} f = \Gamma^0(f)$.
Finally, if $x^k$ lies under $\Pi$, then $\mathbf{v}$ supports $\Gamma_+(f)$ along the vertex $x^k$.
\end{proof}

\section{Proof of the main theorem}

Let, as in Theorem \ref{tw_main}, $f : (\mathbb{C}^3, 0) \rightarrow (\mathbb{C}, 0)$ be an isolated,
$\mathcal{K}\!\!$-non-degenerate singularity, $\Gamma (f)$ -- its Newton boundary,
and $\Gamma^2 (f) \setminus E (f) \neq \varnothing$. Recall that the set of
$\mathcal{K}\!\!$-non-degenerate holomorphic functions with a given finite support
is non-empty and Zariski open in the space of coefficients of the support (cf.
{\cite{Kou76}}, {\cite{Oka79}}); in particular, any such function with support
equal to $\Gamma^0 (f)$ is, by Proposition \ref{BrOl}, an isolated singularity
(and has the same Newton boundary, $\Gamma (f)$). This, together with Theorem
\ref{Brzost}, allows us to assume that $\tmop{supp} f = \Gamma^0 (f)$ and the
coefficients of $f$ are generic (for the monomials in $\tmop{supp} f$).
In particular, $f$ is a polynomial. 
Thus,
% after all the above reductions,
we may and will assume in the sequel that $f :
(\mathbb{C}^3, 0) \rightarrow (\mathbb{C}, 0)$ satisfies the following
conditions: \nopagebreak
\begin{enumerateroman}
  \item \label{ass_1}$f$ is an isolated singularity,
  
  \item $f$ is $\mathcal{K}\!\!$-non-degenerate,
  
  \item \label{ass_3}$\Gamma^2 (f)\setminus E(f) \neq \varnothing$,
  
  \item \label{ass_4}$\tmop{supp} f = \Gamma^0 (f)$,
  
  \item \label{ass_5}the coefficients of $f$ are generic.
%\suspend{enumerateroman}
\end{enumerateroman}
 
%\resume{enumerateroman}
  %\item \label{ass_6}$E(f) = \varnothing$.

We want to prove the equality
\begin{equation}
  \mathcal{L}_0 (f) = \max_{S \in \Gamma^2 (f) \setminus E (f)} m (S) - 1.
  \label{1}
\end{equation}
Denote by $m (f)$ the maximum on the right-hand side of the above formula. The
inequality $\mathcal{L}_0 (f) \leqslant m (f) - 1$ was established by the
third-named author in {\cite{Ole13}}. We will prove the inverse inequality,
$\mathcal{L}_0 (f) \geqslant m (f) - 1$.

From formula {\eqref{wyk_par}}, to justify this inequality, it suffices to find an analytic path $0 \neq \Phi (t) =
(\varphi_1 (t), \varphi_2 (t), \varphi_3 (t)) : (\mathbb{C}, 0) \rightarrow
(\mathbb{C}^3, 0)$ for which
\begin{equation}
  \frac{\tmop{ord} (\nabla f \circ \Phi)}{\tmop{ord} \Phi} \geqslant m (f) - 1.
  \label{12}
\end{equation}
%The idea is that such a $\Phi$ should be somehow connected with
%proximity faces, those realizing $m (f)$.
%(such proximity faces do exist because of Lemma and restriction \ref{ass_3}.

By restriction \eqref{ass_3}, every proximity face of $f$ is non-exceptional (see Proposition \ref{exist_prox}). Hence, by Lemma \ref{le_geo}, $m (f)=\max\left(m(T)_{\omega}\right)$, where $T$ runs through the set
of the faces proximate for the axes $0 \omega$,  $\omega \in \{ x, y, z \}$.
%$We will seek $\Phi(t)$ satisfying \eqref{12} related to proximity faces.
%The idea is that such a $\Phi$ is connected with
%proximity faces, those realizing $m (f)$. 
Fix any proximity face $S$; say, it is
proximate for the axis $0 \nospace x$.

%Let $f_S$ be the quasihomogeneous
%polynomial associated with $S$.
% According to the reductions
% {\eqref{ass_4}}-{\eqref{ass_5}} made above, the
% points of $\tmop{supp} f_S$ are exactly the vertices of $S$ and $f_S$ has
% generic coefficients.
Because of the reductions
{\eqref{ass_4}}-{\eqref{ass_5}} made above, we can use Theorem \ref{key}
to get a suitable $\Phi$ for $S.$ 
Let us call its initial exponent vector $\mathbf{v}$, with coordinates
{$(v_x, v_y, v_z) \in (\mathbb{N}_+ \cup \{+\infty\})^3$.}
% The face $S$ being proximate for $0
% \nospace x$, we get $m (S) = m (S)_x = \frac{l_S}{v_x}$, where $l_S$ is the
% $\mathbf{v}_S$-order of $f$, and $\tmop{ord} \Phi = \tmop{ord} \varphi_1 =
% v_x$.
We have $\varphi_1 \neq 0$ so $v_x=\tmop{ord} \varphi_1 \neq +\infty$.
What is more, $\mathbf{v}$ supports $\Gamma_+(f)$ along a subface of $S$
whose affine hull intersects the axis $0 x$. Hence,
 $m (S)_x = \frac{l}{v_x}$, where $l=l_\mathbf{v}$ is the
$\mathbf{v}$-order of $f$.
As $f$ is $\mathcal{K}\!\!$-non-degenerate, we get $\tmop{ord} ( \frac{\partial f}{\partial x} \circ
\Phi ) = \tmop{ord} (f \circ \Phi) - \tmop{ord} \varphi_1$.
%(actually, there is equality here for $f$ is $\mathcal{K}\!\!$-non-degenerate on $S$)
Using these observations, we can write$\label{rachunek}${\nopagebreak}
% \multbox
% \begin{align*}\label{rachunek}
  % \mathcal{L}_0 (f) \geqslant \frac{\tmop{ord} (\nabla f \circ \Phi)}{\tmop{ord}
  % \Phi} = \frac{\tmop{ord} \left( \frac{\partial f}{\partial x} \circ \Phi
  % \right)}{\tmop{ord} \Phi}  \geqslant \frac{\tmop{ord} (f \circ \Phi) -
  % \tmop{ord} \varphi_1}{\tmop{ord} \varphi_1} &\geqslant \frac{\tmop{ord} (f_S \circ
  % \Phi)}{\tmop{ord} \varphi_1} - 1\\
  % & = \frac{l_S}{v_x} - 1 = m (S)_x - 1. %%%\tag*{\proofbox}																				  
% \end{align*}
% \emultbox
\begin{align*}\label{rachunek}
  \mathcal{L}_0 (f) \geqslant \frac{\tmop{ord} (\nabla f \circ \Phi)}{\tmop{ord}
  \Phi} = \frac{\tmop{ord} \left( \frac{\partial f}{\partial x} \circ \Phi
  \right)}{\tmop{ord} \Phi}  \geqslant \frac{\tmop{ord} (f \circ \Phi) -
  \tmop{ord} \varphi_1}{\tmop{ord} \varphi_1} &\geqslant \frac{l}{v_x} - 1 = m (S)_x - 1.																			  
\end{align*}
As $S$ was an arbitrarily fixed proximity face, we infer that $\mathcal{L}_0 (f) \geqslant m(f)-1$.
This ends the proof of the theorem. \hspace*{\fill}\proofbox

% \medskip

% It remains to give
% \begin{proof*}[of Corollary \ref{le_remove}]
% By assumption, $\varnothing\neq\Gamma^2 (f)\setminus E(f)$. Hence, by Proposition \ref{exist_prox},
% there exist non-exceptional proximity faces for each of the axis. Since $\Gamma^2 (f)\setminus E(f)=\Gamma^2 (g) \setminus E(g)$, the families of 
% proximity faces for $f$ and $g$ are the same. Thus, by Lemma \ref{le_geo} and the main theorem, we get the assertion.
% \hspace*{\fill}\proofbox
% \end{proof*}

\section{Concluding remarks.}

%\begin{itemize}
\hspace{0.75em}1.\hspace{1em} The main theorem, together with Comment \ref{com_2}, gives an effective method of
calculation of the \L ojasiewicz exponent of $\mathcal{K}\!\!$-non-degenerate surface
singularities. For instance, most singularities from Arnold's list are exactly such ones.
To calculate the \L ojasiewicz exponent of a $\mathcal{K}\!\!$-non-degenerate
surface singularity, it suffices to detect non-exceptional faces in the Newton boundary
$\Gamma(f)$.
If we do not find any, then we get an immediate answer (see Comment \ref{com_2} on page \pageref{com_2}).
%In the opposite case, \tmem{viz.} if there are some non-exceptional faces in $\Gamma(f)$,
If we do find some such faces, we must just compute their intersections (precisely,
the intersections of the planes containing them) with the coordinate axes.
It suffices, however,
to compute
%these intersections only for proximate faces of each axis. %(see the proof of Corollary \ref{le_remove}).
these intersections using just one (arbitrarily chosen) proximaty face for each axis
(it does exist by Proposition \ref{exist_prox}).
Namely
%{\sout{(which are non-exceptional by
%definition)}}

\begin{corollary} \label{wnios1}
Let $f$ be a
  $\mathcal{K}\!\!$-non-degenerate isolated surface singularity possessing non-exceptional
  faces and $S_{x},S_{y},S_{z}$ be any
proximate faces for  axes $0x,0y,0z$, respectively. Then%
\[
\mathcal{L}(f)=\max(m(S_{x})_{x},m(S_{y})_{y},m(S_{z})_{z})-1.
\]

\end{corollary}
\begin{proof}
%By Proposition \ref{exist_prox},
%there exist non-exceptional proximity faces for each of the axis.
%Hence, $S_{x},S_{y},S_{z}$ are non-exceptional, too.
Using Lemma \ref{le_geo} and formula \eqref{formula2}, we get the assertion.
\end{proof}

In particular, if non-exceptional faces of $\Gamma(f)$ touch all the axes, say
in the points $x^{m},y^{n},z^{k}$, then they are obviously proximate and
\[
\mathcal{L}(f)=\max(m,n,k)-1.
\]

\hspace{0.75em}2.\hspace{1em} An application of the main result to the problem of constancy of the {\L}ojasiewicz
exponent in non-degenerate deformations of surface singularities shall be the
subject of a next article.

%\end{itemize}

{\affiliationone{Szymon Brzostowski, Tadeusz Krasi{\'n}ski, and Grzegorz Oleksik\\
Faculty~of~Mathematics and~Computer~Science\\
University of {\L}{\'o}d{\'z}\\
ul. Banacha 22, 90-238 {\L}{\'o}d{\'z}\\
Poland\\
{\email{szymon.brzostowski@wmii.uni.lodz.pl\\
tadeusz.krasinski@wmii.uni.lodz.pl\\
grzegorz.oleksik@wmii.uni.lodz.pl}}}}				 

\end{document}